\documentclass[12pt]{amsart}

\newtheorem{theorem}{Theorem}[section]
\newtheorem{lemma}[theorem]{Lemma}

\theoremstyle{definition}
\newtheorem{definition}[theorem]{Definition}

\theoremstyle{remark}
\newtheorem{remark}[theorem]{Remark}
\newtheorem{assumption}{Assumption}[section]

\numberwithin{equation}{section}

\usepackage{mathtools}
\usepackage{amsmath,amsfonts,amssymb}  
\usepackage{appendix}
\usepackage{enumitem}
\usepackage{hyperref}
\usepackage{cleveref}
\usepackage{amsrefs}

\crefname{Proposition}{proposition}{propositions}
\Crefname{Proposition}{Proposition}{Propositions}

\crefname{Corollary}{corollary}{Corollaries}
\Crefname{Corollary}{Corollary}{Corollaries}

\crefname{Result}{result}{results}
\Crefname{Result}{Result}{Results}

\crefname{Lemma}{lemma}{lemmas}
\Crefname{Lemma}{Lemma}{Lemmas}

\crefname{Remark}{remark}{remarks}
\Crefname{Remark}{Remark}{Remarks}

\crefname{Definition}{definition}{definitions}
\Crefname{Definition}{Definition}{Definitions}

\crefname{Example}{example}{examples}
\Crefname{Example}{Example}{Examples}

\newcommand{\Cov}{\mathsf{Cov}}

\newcommand{\Var}{\mathsf{Var}}
\newcommand{\E}{\mathsf{E}}

\newcommand{\Eof}[1]{\E\left(#1\right)}
\newcommand{\prob}{\mathsf{P}}

\newcommand{\cF}{{\mathcal F}}

\newcommand{\vk}{\varkappa}
\newcommand{\vf}{\varphi}
\newcommand{\ve}{\varepsilon}
\newcommand{\mbR}{{\mathbb R}}
\newcommand{\mbN}{{\mathbb N}}
\newcommand{\R}{\mbR}
\newcommand{\N}{\mbN}

\newcommand{\mbH}{{\mathbb H}}

\newcommand{\lplim}[1]{\mathop{L_p\text{-}\lim}\limits_{#1}\,}

\newcommand{\myExp}[1]{\exp\left(#1\right)}
\newcommand{\differential}[1]{\mathrm{d}#1}

\newcommand{\norm}[1]{\| #1 \|}

\newcommand{\vol}[1]{\mathsf{Vol}\left( #1 \right)}

\begin{document}

\author[O. Izyumtseva]{Olga Izyumtseva}
\address{Olga Izyumtseva, 
School of Mathematical Sciences, 
University of Nottingham, 
University Park, 
Nottingham NG7 2RD, 
United Kingdom
}
\thanks{Olga Izyumtseva was supported by the British Academy through grant number RaR{\textbackslash}100741, and in part by the British Academy, Cara and the Leverhulme Trust through grant LTRSF24{\textbackslash}100014.}
\email{Olga.Iziumtseva1@nottingham.ac.uk}
\urladdr{https://sites.google.com/view/olga-izyumtseva/home}

\author[W. R. KhudaBukhsh]{Wasiur R. KhudaBukhsh}
\address{Wasiur R. KhudaBukhsh, 
School of Mathematical Sciences, 
University of Nottingham, 
University Park, 
Nottingham NG7 2RD, 
United Kingdom
}
\thanks{Wasiur R. KhudaBukhsh was partially supported by the Engineering and Physical Sciences Research Council (EPSRC) of UK Research and Innovation
(UKRI) [grant number EP/Y027795/1]. 
}
\email{wasiur.khudabukhsh@nottingham.ac.uk}
\urladdr{https://www.wasiur.xyz/}

\subjclass[2020]{Primary 60G15, 60J55; Secondary 60H40}

\keywords{Self-intersection local times, Volterra Gaussian bridge, local times, Gaussian processes}

\title[Volterra Gaussian Bridges]{Self-intersection local times of Volterra Gaussian bridges with large end-to-end distance}

\begin{abstract}
  Self-intersection local times serve as the Hamiltonian in the Gibbs measure of continuous polymer models, for which the end-to-end distance is a key quantity. We study multiple self-intersection local times of Volterra Gaussian bridges, \emph{i.e.}, Volterra Gaussian processes conditioned on their final location $a$, as $\|a\|\to\infty$. For $d$-dimensional $(2,\zeta)$-locally nondeterministic processes with $\zeta\in(0,2/d)$, we prove that these local times exist in every $L^p$. Under two further conditions on the kernel, the $p$-th moment of the $m$-multiple self-intersection local time of the bridge is of exact order $\|a\|^{-\gamma}$ for every integer $p\geqslant1$, with $\gamma=p(m-1)(2-\zeta d)/(2-\zeta)$. In the plane, for kernels $k(t-s)$ with $k(0)\neq0$ and under strong local nondeterminism, we restrict the double self-intersection local time to $\{t_2-t_1\geqslant\|a\|^{-\alpha}\}$, $\alpha>0$, and establish a phase transition: its expectation decays like $\exp(-c\|a\|^{2-\alpha})$, converges to an explicit positive limit, or grows logarithmically, for $\alpha<2$, $\alpha=2$ and $\alpha>2$ respectively.
\end{abstract}

\maketitle

\tableofcontents

\section{Introduction}
\label{sec:Intro}
Let $(\Omega, \cF, \prob)$ be a (large enough) complete probability space.
In this paper, we study the asymptotics of moments of self-intersection local times for Volterra Gaussian bridges. The interest in this question arises from the study of mathematical polymer models and their geometric characteristics, in particular the end-to-end distance \cite{caravenna2011lectures}, \cite{desCloizeauxJannink1990}, \cite{hollander2009random}, \cite{teraoka2002polymer}, \cite{vdHofstadKonig2001}, \cite{van2001self}. See also \cite{Hofstader1997CLTEdwards}. It is well known that a continuous  polymer model can be constructed from a continuous stochastic process \cite{BockEtAl2015}, \cite{chen2009charged}, \cite{DoiEdwards1986}, \cite{GrothausEtAl2011}, \cite{Kusuoka1985}, \cite{MadrasSlade1993}, \cite{Westwater1980}, \cite{Westwater1981}, \cite{westwater1982edwards}. Because of the excluded volume effect of real polymers \cite{Edwards1965Polymers}, the polymer measure should assign small mass to trajectories with a large number of self-intersections. This can be achieved by first defining the self-intersection local time of the stochastic process and then taking as the polymer measure the Gibbs measure whose Hamiltonian is the self-intersection local time. 

The self-intersection local time for an $\mbR^d$-valued stochastic process $\{u(t):\ t\in[0,1]\}$ is formally defined as
\begin{align}
\label{eq:SILT_formal}
T_{m}^{u}\coloneq \int_{\Delta_m}\prod^{m-1}_{i=1}\delta_0(u(t_{i+1})-u(t_i))\differential{t_1}\cdots \differential{t_m},
\end{align}
where $$\Delta_m=\{(t_1,\ldots, t_m)\in[0,1]^m:\ 0\leqslant t_1\leqslant\ldots\leqslant t_m\leqslant 1\},\ m\geqslant 2,$$ and $\delta_0$ is the $d$-dimensional Dirac delta function at zero. 
The quantity $T_{m}^{u}$ registers times $t_1\leqslant t_2\leqslant \cdots \leqslant t_m$ such that $u(t_1)=u(t_2)=\cdots=u(t_m)$, \emph{i.e.}, $T_{m}^{u}$ measures the amount of time the process $u$ spends in small neighbourhoods of its self-intersection points of multiplicity $m$. To give a rigorous meaning to the formal expression \eqref{eq:SILT_formal}, consider the family of approximations
$\{T_{\ve,m}^{u} : \ve >0\}$ defined as 
$$
T_{\ve,m}^{u}\coloneq \int_{\Delta_m}\prod^{m-1}_{i=1}f_\ve(u(t_{i+1})-u(t_i))\differential{t_1}\cdots \differential{t_m},\ \ve>0,
$$
where the Gaussian density 
$$
f_{\ve}(y)\coloneq \frac{1}{(2\pi\ve)^{\frac{d}{2}}}\myExp{-\frac{\|y\|^2}{2\ve}},\ y\in\mbR^d, 
$$
as we know, converges to the Dirac delta function $\delta_0$ as $\ve \to 0+$.

\begin{definition}
\label{def:SILT} The random variable $T^{u}_m \coloneq \lplim{\ve\to0} T^{u}_{\ve,m}$ is called the $m$-multiple self-intersection local time of the stochastic process $u$ if the limit exists in $L^p(\Omega, \cF, \prob)$ for some $p\geqslant 2.$ 
\end{definition}

Self-intersection local times are an important geometric characteristic of non-smooth stochastic processes. The precise meaning of this sentence is found in the beautiful work of Le Gall \cite{LeGallWienersausage} in which he showed that the random variable that appears in the $m$-th term in the asymptotic expansion of the area of the Wiener sausage is precisely the $m$-multiple self-intersection local time. For a planar Wiener process, self-intersection local times were constructed in \cite{Varadhan1969,rosen1986renormalized,Dynkin1988}, see also \cite{LeGall1985,geman1984local}. For the fractional Brownian motion, self-intersection local times were studied in \cite{Rosen1987IntersectionFBM,HuNualart2005,HuNualart2007,HuNualartSong2008,NualartOrtizLatorre2007,OliveiraDaSilvaStreit2011,WuXiao2010,Hu2001fBmChaos,Gu2026derivativeSILTfBM}. For compactly perturbed Wiener processes, self-intersection local times were constructed in \cite{Dorogovtsev2011regularization}, for generalised Brownian bridges in \cite{DorogovtsevIzyumtseva2015bridges}. Recently, they have been studied for Volterra Gaussian processes in \cite{Izyumtseva2026VGP}, \cite{Izyumtseva2026StochascticFlowInteraction}.

In this paper, we study self-intersection local times for Volterra Gaussian bridges and their moments, which appear as conditional moments of self-intersection local times of Volterra Gaussian processes given the end-to-end distance. Volterra Gaussian processes form a large class of Gaussian processes containing several important examples, such as the Wiener process, the Brownian bridge (with the kernel $k(t,s)=\frac{1-t}{1-s}1_{[0,t]}(s)$), and the fractional Brownian motion (with the Molchan--Golosov kernel, see \cite{MolchanGolosov1969}, \cite{NorrosValkeilaVirtamo1999}, \cite{DecreusefondUstunel1999}). 
Let $\{w(t):\ t\in [0,1]\}$ be a one-dimensional Wiener process.  
\begin{definition}
\label{def:Volterra} A real-valued centred Gaussian process $\{u(t): \ t\in[0,1]\}$ is called a Volterra Gaussian process if it admits the representation
$$
u(t) = 
\int^t_0k(t,s)\differential{w(s)}, \quad t\in[0,1],$$ 
where $k\in L^2([0,1]^2)$ is a Volterra kernel, \emph{i.e.}, 
$k(t,s)=0$ for all $s>t$, and $$\sup_{t\in[0,1]}\int^t_0k^2(t,s)\differential{s}<\infty.$$
\end{definition}
It will be convenient to work with the white noise representation of Volterra Gaussian processes. Therefore, let us recall the definition of white noise in a Hilbert space (also called an isonormal Gaussian process, see \cite[Chapter 1]{Janson1997GHS}, \cite[Section 4]{Lifshits2012LecturesOnGaussianProcesses}, or \cite[Chapter 1]{Dalang2026SPDE}). Let $\mbH$ be a real separable Hilbert space with inner product $(\cdot,\cdot)$. 

\begin{definition}
  \label{def:white_noise} 
  A centred (mean zero) Gaussian random functional $\xi$ on $\mbH$ such that the mapping $h\mapsto \xi(h)$ is linear, and $\E({\xi(h)\xi(g)})=(h,g)$ for all $ h,\ g\in \mbH$ is called a white noise in $\mbH.$
\end{definition}
In this paper, we will use the dual pairing notation $(h,\xi)$ for $\xi(h)$.
Note that if $\xi$ is a white noise in $L^2([0,1])$ generated by the Wiener process $w$ from Definition \ref{def:Volterra}, then the Volterra representation of the process $u$ takes the form
$$
u(t)=(k(t,\cdot)1_{[0,t]},\xi),\ t\in[0,1].
$$

In general, a Volterra Gaussian process is neither a martingale nor a Markov process, but its representation as a stochastic integral of the Volterra kernel $k$ with respect to a Wiener process allows us to associate each Volterra Gaussian process with its kernel and to study the properties of the process via the properties of the kernel. Different kernels give processes with short- and long-range dependence, which are widely used in turbulence, finance, energy markets, medicine, and machine learning  (\emph{e.g.}, see \cite{BarndorffNielsen2004Spatiotemporal}, \cite{BarndorffNielsen2013Energy}, \cite{BarndorffNielsen2014Electricity}, \cite{BarndorffNielsen2007Ambit}). 
In this paper, we consider two classes of Volterra Gaussian processes satisfying two specific local-nondeterminism  conditions. The first class satisfies the following version of the $(\beta,\zeta)$-local nondeterminism condition of \cite[Section 4.1]{harang2022regularity} with $\beta=2$.

 \begin{definition}
 \label{def:LND}
  If the kernel $k$ of a Volterra Gaussian process $u$ satisfies 
   \begin{align}
\label{eq:LND} 
 \inf_{t\in(0,1]}\ \inf_{s\in[0,t)}\frac{1}{(t-s)^{\zeta}}\int^t_sk^2(t,r)\differential{r}>0
  \end{align}
  for some $\zeta>0,$ then the process $u$ is said to be $(2,\zeta)$-locally nondeterministic.
  \end{definition} 
For $s\in[0,1]$, put $\cF^w_s=\sigma(w(r),\ r\leqslant s),$ the filtration generated by the Wiener process $w$. Since $$u(t)=\int_0^sk(t,r)\differential{w(r)}+\int_s^tk(t,r)\differential{w(r)},$$ where the first summand is $\cF^w_s$-measurable and the second one is independent of $\cF^w_s$, we have $\int_s^tk^2(t,r)\differential{r}=\Var(u(t)\mid\cF^w_s)$. Thus, the condition in  \eqref{eq:LND} means that $\Var(u(t)\mid\cF^w_s)\geqslant c\,(t-s)^\zeta$ for all $0\leqslant s<t\leqslant 1$. 

  The second class of Volterra Gaussian processes we consider in this paper satisfies the strong local nondeterminism condition introduced in \cite{Dorogovtsev2011regularization}. The strong local nondeterminism property implies Berman's local nondeterminism property from \cite{Berman1973LND} (see also \cite{Berman1989LNDgeneral}, \cite{Berman1991Selfintersections} and \cite[Definition 4.3]{Izyumtseva2026VGP}).
  Let $x(t)=(g(t),\xi),\ t\in[0,1]$, be a Gaussian process in $\mbR$ defined via the function $g\colon[0,1]\to L^2([0,1])$ and a white noise $\xi$ in $L^2([0,1])$. For $0\leqslant s\leqslant t\leqslant 1$,  write $\Delta g(s,t)=g(t)-g(s)$, and for $(t_1,\ldots,t_m)\in\Delta_m$,  abbreviate $\Delta g(t_i)\coloneq\Delta g(t_i,t_{i+1})$ and $\Delta t_i\coloneq t_{i+1}-t_i$, $i=1,\ldots,m-1$. Let $G(e_1,\ldots,e_n)$ denote the Gram determinant of elements $e_1,\ldots,e_n$ of a Hilbert space. Consider 
  $G(\Delta g(t_1),\ldots,\Delta g(t_{m-1})),\ (t_1,\ldots,t_m)\in\Delta_m$, and assume that $G(\Delta g(t_1),\ldots,\Delta g(t_{m-1}))\neq 0$ for all $0\leqslant t_1<\ldots<t_m\leqslant 1$.

  \begin{definition}[Strongly locally nondeterministic process \cite{Dorogovtsev2011regularization}]
\label{def:strongLND} The process $x$ (or, equivalently, the function $g$) is said to be strongly locally nondeterministic if for every $m\geqslant 2$, every subset $M\subseteq \{1,\ldots, m-1\}$ and every fixed configuration of the remaining increments $\Delta t_i$, $i\notin M$, we have 
$$
G(\Delta g(t_1),\ldots,\Delta g(t_{m-1}))\sim G(\Delta g(t_i),\ i\notin M)\prod_{i\in M}\|\Delta g(t_i)\|^2, 
$$
as $\max_{i\in M}\Delta t_i\to0.$
\end{definition}
The strong local nondeterminism condition implies that the increments of the process $x$ over small time intervals are asymptotically independent, as are the increments of the Wiener process.  We note that various forms of strong local nondeterminism exist in the literature. 
It is, in fact, a common tool used in the study of sample path properties of stochastic processes. For instance, see \cite{Xiao2009Anisotropic} for their extensive use in deriving sample path properties of anisotropic Gaussian fields.

The notion of strong local nondeterminism is also related to the ``crinkled arc'' problem in a Hilbert space (see \cite[Problems 5--6, pp. 5--6]{Halmos1982Hilbert}), the ideas of which date back to Kolmogorov's Wiener spiral \cite{Kolmogoroff1940WienerSpiral,Kolmogoroff1940Kurven}, and ``screw lines'' of von Neumann and Schoenberg \cite{vonNeumann1941ScrewLine}. A simple (one-to-one) continuous curve $g\colon[0,1]\to\mbH$ in a Hilbert space $\mbH$ is called a crinkled arc if every two non-overlapping chords of the curve are orthogonal. Kolmogorov's spirals, and screw lines of von Neumann and Schoenberg demand stationarity as opposed to orthogonality. When $\mbH=L^2([0,1])$, it is elementary to check that the curve $g(t)=1_{[0,t]}$ is crinkled. It is a rather remarkable fact in functional analysis, originally discovered by Johnson in \cite{Johnson1970CrinkledArc} (see also \cite{Johnson1971HilbertSpaceProblemFour,Vitale1975CrinkledArc}; it is believed that Kolmogorov's congruence theorems about spirals remained unnoticed in the functional analysis literature) that any two crinkled arcs are ``essentially the same'' in the following precise sense. 

\begin{theorem}[{\cite[Problem 6]{Halmos1982Hilbert}; \cite{Johnson1970CrinkledArc}; see also \cite[Theorem 2.2]{Ettinger2008LinearNonlinearAcquisition}}]
  \label{thm:crinked_arc_uniqueness}
  Let $f: [0, 1] \to \mbH_1$, and $g: [0, 1] \to \mbH_2$ be two crinkled arcs in Hilbert spaces $\mbH_1$ and $\mbH_2$. Then, there exist elements $a \in \mbH_1$, and $b \in \mbH_2$ (for translation), a positive constant $c >0$ (for scaling), and a partial isometry $U: \mbH_1 \to \mbH_2$ such that
  \begin{align*}
    U(f(\rho(t)) - a ) = c g(t) - b, \quad t \in [0, 1],
  \end{align*}
  for some reparametrisation (an increasing homeomorphism) $\rho: [0, 1] \to [0, 1]$.
\end{theorem}

Therefore, in the language of Halmos, a strongly locally nondeterministic curve $g \in C([0, 1], L^2([0, 1])),$ the space of continuous functions from $[0, 1]$ to $L^2([0, 1])$, is ``crinkled'' at small scales in the sense that the increments of $g$ over small time intervals are asymptotically orthogonal, with $M = \{1, 2, \ldots, m-1\}$. The Volterra Gaussian process $u$ in Definition~\ref{def:Volterra} corresponds to the curve $g(t)(\cdot) = k(t, \cdot) 1_{[0, t]}(\cdot)$ in $L^2([0, 1])$. For $0\le s < t \le 1$, the increments of $g$ can be written as
\begin{align*}
  g(t) - g(s) = k(t, \cdot) 1_{(s, t]}(\cdot) + (k(t, \cdot) - k(s, \cdot)) 1_{[0, s]}(\cdot).
\end{align*}
Therefore, it is easy to see that  if the history part of the increment, $(k(t, \cdot) - k(s, \cdot)) 1_{[0, s]}(\cdot)$, disappears (for instance, if $k(t, r) = h(r)$ for some function $h\in L^2([0, 1])$ with $\int_s^t h^2(r)\differential{r}>0$ for all $0\le s < t\le 1$), then the curve $g$ is exactly crinkled. This precisely corresponds to the case of the (scaled) Wiener process. Under the strong local nondeterminism condition, this property of being crinkled is desired only at small timescales. This rather strong assumption leads to an interesting trichotomy (indeed a phase transition) in the asymptotic behaviour of the conditional moments of the double self-intersection local times in the planar case (Theorem~\ref{thm:5.1}). Before concluding this section, we summarise our main mathematical contributions.

\subsection{Main contributions}
\label{subsec:contributions}
In polymer models associated with a Volterra Gaussian process $u$,  a central geometric characteristic is the end-to-end vector $u(1)$. For such models, the conditional moments $\Eof{(T^u_m)^p\mid u(1)=a}$ describe how the number of self-intersections of a polymer depends on how far it is stretched. To the best of our knowledge, asymptotics of such conditional moments have not been studied in the literature for general Volterra Gaussian processes beyond the case of the Wiener process. Our main contributions in this paper are as follows.

\begin{enumerate}[label=\textup{(\arabic*)},leftmargin=*]
\item \emph{Volterra Gaussian bridges and existence of their self-intersection local times.} We work with the representation \eqref{eq:VGB} of Volterra Gaussian bridges obtained by projecting the kernel onto the orthogonal complement of $k(1,\cdot)$. For bridges of Gaussian processes and their orthogonal and canonical representations, see \cite{GasbarraSottinenValkeila2007,SottinenYazigi2014}. We prove that, for $(2,\zeta)$-locally nondeterministic kernels with $\zeta\in(0,\frac2d)$, their $m$-multiple self-intersection local times exist in $L^p(\Omega,\cF,\prob)$ for every $p\geqslant1$ and every $m\geqslant2$ (Theorem~\ref{thm:SILTVGB}).  The key tool is a lower bound for Gram determinants of integer combinations of values of $u$ (Lemma~\ref{lem:Gram}). 

\item \emph{Sharp polynomial decay of all conditional moments.} Under $(2,\zeta)$-local nondeterminism together with the conditions \eqref{eq:drift} and \eqref{eq:upper}, we prove the two-sided estimate
\[
\Eof{(T^u_m)^p\mid u(1)=a}\asymp\norm{a}^{-\gamma_{p,m}},\qquad \gamma_{p,m}=p(m-1)\,\frac{2-\zeta d}{2-\zeta},\qquad \norm{a}\to\infty,
\]
for all $p\in\mbN$, $m\geqslant2$ and $d\geqslant1$ (Theorem \ref{thm:4.2}). The exponent is optimal. For the fractional Brownian motion with $Hd<1$, the exponent equals $p(m-1)\frac{1-Hd}{1-H}$. In dimension one, it equals $p(m-1)$ for all $H\in(0,1)$. That is, the memory of the process does not affect the rate of decay. The proof rests on an estimate for integrals over products of simplices (Lemma~\ref{lem:gaps}), which is of independent interest.
\item \emph{The critical planar case and a phase transition.} For planar bridges generated by kernels of the form $k(t-s)$ with $k(0)\neq0$,  we first describe the Rosen renormalised self-intersection local times in Theorem~\ref{thm:Fourier_Wiener}. For the self-intersection local time over time domains satisfying $\{t_2-t_1\geqslant\norm{a}^{-\alpha}\}$, we obtain a sharp phase transition at $\alpha=2$ (Theorem \ref{thm:5.1}): the conditional expectation decays like $\exp(-c\norm{a}^{2-\alpha})$ for $\alpha<2$, converges to an explicit positive limit for $\alpha=2$, and grows like $\frac{\alpha-2}{2\pi k^2(0)}\log\norm{a}$ for $\alpha>2$. The coefficient of the logarithm depends on the kernel only through $k(0)$. 
\end{enumerate}

\subsection{Structure of the paper}
The rest of the paper is organised as follows. In Section \ref{sec:SILTVGB}, we describe the Volterra Gaussian bridges and prove the existence of their self-intersection local times in $L^p(\Omega, \cF, \prob)$. In Section \ref{sec:RR}, we construct the Rosen renormalisation of self-intersection local times for planar strongly locally nondeterministic Volterra Gaussian bridges. In Section \ref{sec:ACM}, we discuss the asymptotics of conditional moments of self-intersection local times of Volterra Gaussian processes satisfying the $(2,\zeta)$-local nondeterminism condition for large end-to-end distances. In Section \ref{sec:asymptotics2}, we describe the asymptotics of conditional moments of self-intersection local times of strongly locally nondeterministic Volterra Gaussian processes for large end-to-end distances.
\section{Volterra Gaussian bridges and the existence of their self-intersection local times} 
 \label{sec:SILTVGB} 
Let $w_1,\ldots, w_d$ be independent one-dimensional Wiener processes generating white noises $\xi_1,\ldots,\xi_d$ in $L^2([0,1])$.
Consider the $\mbR^d$-valued Volterra Gaussian process 
$$
u(t)=(u_1(t),\ldots, u_d(t)),\ t\in [0,1],
$$ 
where the coordinate $u_i$ is an $\mbR$-valued Volterra Gaussian process generated by the kernel $k$ and the Wiener process $w_i$,  for  $i=1,\ldots,d$. 
\begin{assumption}
    \label{assumption:kernel}
    Throughout the paper we assume that 
there exist constants $L>0$ and $\vk>0$ such that for all $s,t\in[0,1]$,
  \begin{align}
\label{eq:Holder} 
\Eof{(u_1(t)-u_1(s))^2} = 
\int_0^1\big(k(t,r)1_{[0,t]}(r)-k(s,r)1_{[0,s]}(r)\big)^2\differential{r}\leqslant L|t-s|^{\vk}.
 \end{align}
\end{assumption}

Since the random variable $u_1(t)-u_1(s)$ is Gaussian, the inequality in \eqref{eq:Holder} implies $\Eof{|u_1(t)-u_1(s)|^{2n}}\leqslant c_nL^n|t-s|^{n\vk}$ for every $n\in\mbN$. Therefore, by Kolmogorov's continuity criterion (see, e.g., \cite[Theorem I.2.1]{RevuzYor1999CMBMbook}), the stochastic process $u$ has a continuous modification. Henceforth, we work with this continuous modification. Condition \eqref{eq:Holder} holds with $\vk=1$ whenever $k$ is bounded and $|k(t,r)-k(s,r)|\leqslant L|t-s|$ for $0\leqslant r\leqslant s\leqslant t\leqslant 1$ (in particular, for the Wiener process), and with $\vk=2H$ for the fractional Brownian motion with Hurst index $H\in(0,1)$, since in that case the left-hand side of \eqref{eq:Holder} equals $|t-s|^{2H}$. Note that it is important to allow kernels that are unbounded near the diagonal, since a bounded kernel cannot satisfy the $(2, \zeta)$-local nondeterminism condition in \eqref{eq:LND} with $\zeta<1$.

For an $\mbR^d$-valued $(2,\zeta)$-locally nondeterministic Volterra Gaussian process with $\zeta\in (0,\frac{2}{d})$,  there exists a random variable $T^{u}_m$ such that 
$T^{u}_m=\lplim{\ve\to0}T^{u}_{\ve,m}$ for any $p\ge 2$. See \cite[Theorem 2.2]{Izyumtseva2026StochascticFlowInteraction} (with the weight function $\rho\equiv1$) and  Remark~\ref{rem:unconditional} below. 
In this paper, we are interested in the asymptotics of $\Eof{(T^u_m)^p\mid u(1)=a}$ for  $p\in\mbN$ as $\norm{a}\to\infty.$ 

Denote by $P$ the orthogonal projection in $L^2([0,1])$ onto the linear span of the element (function) $k(1,\cdot)1_{[0,1]}$, and put
$$
\vf(t)=\frac{(k(1,\cdot),k(t,\cdot)1_{[0,t]})}{\norm{k(1,\cdot)}^2}=\frac{\int^t_0k(1,r)k(t,r)\differential{r}}{\int^1_0k^2(1,r)\differential{r}},\ t\in[0,1].
$$
Consider the Volterra Gaussian bridge $\{\eta^a(t): \ t\in[0,1]\}$ in $\mbR^d$ from $0$ at time $0$ to $a=(a_1,\ldots,a_d)$ at time $1$, \emph{i.e.}, the Volterra Gaussian process $u$ conditioned to be $a$ at time $1$. Then, the coordinate process $\eta^a_i$ has the following representation: for $t \in [0, 1]$, 
\begin{align}
&{}\eta^a_i(t)=((I-P)k(t,\cdot)1_{[0,t]},\xi_i)+\vf(t)a_i
=u_i(t)-\vf(t)u_i(1)+\vf(t)a_i.
\label{eq:VGB}
\end{align}
For simplicity, let us denote by $\eta(t)\coloneq \eta^0(t),\ t\in[0,1]$, the Gaussian process in $\mbR^d$ with coordinates
\begin{align}
\label{eq:VGBB}
\eta_i(t)=((I-P)k(t,\cdot)1_{[0,t]},\xi_i),\ t\in[0,1],\ i=1,\ldots,d.
\end{align}

\begin{remark}
  \label{rem:regular_conditional}
  By the representation in \eqref{eq:VGB}, $u=\eta+\vf\,u(1)$. Moreover, the process $\eta$ is independent of $u(1)$, because all these random variables are jointly Gaussian and
$$
\Cov(\eta_i(t),u_i(1))=((I-P)k(t,\cdot)1_{[0,t]},k(1,\cdot))=0.
$$
Hence, the law of $\eta^a=\eta+\vf\,a$ is a regular conditional distribution of $u$ given $u(1)=a$ (see, \emph{e.g.}, \cite[Chapter 6]{Kallenberg2002}). 
We show in Lemma \ref{lem:conditional} below that $\Eof{(T^u_m)^p\mid u(1)=a}=\Eof{(T^{\eta^a}_m)^p}$ for every $p\in\mbN$. 
\end{remark}

Our next goal is to prove the existence of self-intersection local times for the Volterra Gaussian bridge $\eta^a$.
\subsection{Existence in $L^p(\Omega, \cF, \prob)$} 
The aim of this subsection is to prove the existence in $L^p(\Omega, \cF, \prob)$ of the random variable
$$
T^{\eta^a}_m=\int_{\Delta_m}\prod^{m-1}_{i=1}\delta_0(\eta^a(t_{i+1})-\eta^a(t_{i}))\differential{t_1}\cdots\differential{t_m}
$$
for a $(2,\zeta)$-locally nondeterministic Volterra Gaussian process $u.$ We do that by proving that the family of approximations $\{T^{\eta^a}_{\ve,m}: \ve >0\}$ given by 
 $$
T^{\eta^a}_{\ve,m}=\int_{\Delta_m}\prod^{m-1}_{i=1}f_{\ve}(\eta^a(t_{i+1})-\eta^a(t_{i}))\differential{t_1}\cdots\differential{t_m}
$$
converges in $L^p(\Omega, \cF, \prob)$ as $\ve\to0$. Such convergence problems very often involve integrability of certain Gram determinants.  Gram determinants of Gaussian random variables will be computed in $L^2(\Omega,\cF,\prob)$. Therefore, the notation $G(Y_1,\ldots,Y_n)$ will denote the determinant of the covariance matrix of $(Y_1,\ldots,Y_n)$. We start with a lower bound for Gram determinants that will be useful for the rest of the paper.

\begin{lemma}
\label{lem:Gram}
Assume $u$ is $(2,\zeta)$-locally nondeterministic, and let $c_\zeta>0$ be such that $\int_s^tk^2(t,r)\differential{r}\geqslant c_\zeta(t-s)^\zeta$ for all $0\leqslant s<t\leqslant 1$. Let $0\leqslant s_1\leqslant s_2\leqslant\ldots\leqslant s_N\leqslant 1$, put $s_0=0$, and let $V=(v_{ij})$ be an $n\times N$ matrix with integer entries and rank $n$. Define $Y_i=\sum_{j=1}^Nv_{ij}u_1(s_j)$, $i=1,\ldots,n$. Then, 
$$
G(Y_1,\ldots,Y_n)\geqslant c_\zeta^{\,n}\prod_{j=1}^N(s_j-s_{j-1})^\zeta.
$$
\end{lemma}
\begin{proof}[Proof of Lemma~\ref{lem:Gram}]
Reversing the order of the columns of $V$ and using the Hermite normal form of integer matrices (see, \emph{e.g.}, \cite[Section 4.1]{Schrijver1986}), we find an integer $n\times n$ matrix $U$ with $\det U=\pm1$ such that, after a suitable permutation of rows, the matrix $V'=UV=(v'_{ij})$ has the following property: there are indices $q_1<q_2<\ldots<q_n$ such that $v'_{iq_i}\neq0$ and $v'_{ij}=0$ for $j>q_i$, $i=1,\ldots,n$. Put $Y'_i=\sum_{j=1}^Nv'_{ij}u_1(s_j)$. Since $(Y'_1,\ldots,Y'_n)^T=U(Y_1,\ldots,Y_n)^T$, we have $G(Y'_1,\ldots,Y'_n)=(\det U)^2G(Y_1,\ldots,Y_n)=G(Y_1,\ldots,Y_n)$.

Fix $i$. The random variables $Y'_1,\ldots,Y'_{i-1}$ and $Y'_i-v'_{iq_i}u_1(s_{q_i})$ are linear combinations of $u_1(s_1),\ldots,u_1(s_{q_i-1})$, which are measurable with respect to $\cF^{w_1}_{s_{q_i-1}}$ (for $q_i=1$,  this $\sigma$-field is trivial). Therefore, 
\begin{align*}
\Var(Y'_i\mid Y'_1,\ldots,Y'_{i-1})&\geqslant\Var(Y'_i\mid\cF^{w_1}_{s_{q_i-1}})=(v'_{iq_i})^2\int_{s_{q_i-1}}^{s_{q_i}}k^2(s_{q_i},r)\differential{r}
\geqslant c_\zeta(s_{q_i}-s_{q_i-1})^\zeta,
\end{align*}
where we used the fact that $v'_{iq_i}$ is a nonzero integer. Consequently, since the Gram determinant of a Gaussian vector is the product of the successive conditional variances (this follows from the Schur complement formula for determinants, see, \emph{e.g.}, \cite[Section 0.8.5]{HornJohnson2013}),
\begin{align*}
G(Y_1,\ldots,Y_n)&=\prod_{i=1}^n\Var(Y'_i\mid Y'_1,\ldots,Y'_{i-1})
\geqslant c_\zeta^{\,n}\prod_{i=1}^n(s_{q_i}-s_{q_i-1})^\zeta\geqslant c_\zeta^{\,n}\prod_{j=1}^N(s_j-s_{j-1})^\zeta,
\end{align*}
where the last inequality holds because $0\leqslant s_j-s_{j-1}\leqslant 1$, for all $j$.
\end{proof}

 \begin{theorem}
  \label{thm:SILTVGB} 
  Suppose that the Volterra Gaussian process $\{u(t):\ t\in[0,1]\}$ in $\mathbb{R}^d$ is $(2,\zeta)$-locally nondeterministic for some $\zeta\in (0,\frac{2}{d})$. Then, for all $a\in\mbR^d$ and all $p\geqslant 1$, there exists a random variable $T^{\eta^a}_m$ such that the family $\{T^{\eta^a}_{\ve,m}: \ve >0\}$ converges to $T^{\eta^a}_m$ as $\ve \to 0$ in $L^p(\Omega, \cF, \prob)$, \emph{i.e.},  
\begin{align*}
T^{\eta^a}_m = \lplim{\ve\to0}T^{\eta^a}_{\ve,m}.
\end{align*}
\end{theorem}
 \begin{proof}[Proof of Theorem~\ref{thm:SILTVGB}]
Since convergence in $L^p(\Omega,\cF,\prob)$ implies convergence in $L^{p'}(\Omega,\cF,\prob)$ for $p'\leqslant p$, we may assume that $p$ is an even integer. Let us check that the family $\{T^{\eta^a}_{\ve,m}:\ \ve>0\}$ is Cauchy in $L^p(\Omega,\cF,\prob)$. For $\ve_1,\ \ve_2>0$, note that 
\begin{align*}
&{}\Eof{(T^{\eta^a}_{\ve_1,m}-T^{\eta^a}_{\ve_2,m})^p}
=\sum^p_{l=0}\binom{p}{l}(-1)^{p-l}\Eof{(T^{\eta^a}_{\ve_1,m})^l(T^{\eta^a}_{\ve_2,m})^{p-l}}.
\end{align*}
Since $\sum^p_{l=0}\binom{p}{l}(-1)^{p-l}=0$, it suffices to check that $\Eof{(T^{\eta^a}_{\ve_1,m})^l(T^{\eta^a}_{\ve_2,m})^{p-l}}$ has a finite limit (as $\ve_1,\ \ve_2\to0$) that does not depend on $l$. We give the details for $\ve_1=\ve_2=\ve$; the general case is identical, since only the variances of the Gaussian kernels change.

Let $S=\{1,\ldots,pm-1\}\setminus\{m,2m,\ldots,(p-1)m\}$ be the set of indices of the increments within the $p$ blocks of $t=(t_1,\ldots,t_{pm})\in\Delta^p_m$. Then, 
\begin{align*}
\Eof{(T^{\eta^a}_{\ve,m})^p}=\int_{\Delta^p_m}\Eof{\prod_{i\in S}f_{\ve}(\Delta\eta(t_{i})+q(t_i,t_{i+1})a)}\differential{t_1}\ldots\differential{t_{pm}},
\end{align*}
where
$
\Delta \eta(t_i)\coloneq\eta(t_{i+1})-\eta(t_{i})
$
and
$$
q(s,t)\coloneq\vf(t)-\vf(s)=\frac{\int^{t}_0k(1,r)k(t,r)\differential{r}-\int^{s}_0k(1,r)k(s,r)\differential{r}}{\norm{k(1,\cdot)}^2}.
$$
Denote $\Delta g(t_i)=k(t_{i+1},\cdot)1_{[0,t_{i+1}]}-k(t_{i},\cdot)1_{[0,t_{i}]}$. Let $A_t$ be the Gram matrix of $\{(I-P)\Delta g(t_i)$, $i\in S\}$, and let $$G_t=\det A_t=G((I-P)\Delta g(t_i),\ i\in S).$$
The $d$ coordinates of the Gaussian vector $(\Delta\eta(t_i))_{i\in S}$ are independent, and each of them has covariance matrix $A_t$. We show below that $G_t>0$ for almost every $t$. Note that 
 \begin{align*}
&{} 
\Eof{\prod_{i\in S}f_{\ve}(\Delta\eta(t_{i})+q(t_i,t_{i+1})a)}\\
&{}=\int_{(\mbR^{d})^S}\prod_{i\in S}f_{\ve}(y_i+q(t_i,t_{i+1})a)\frac{1}{(2\pi)^{\frac{p(m-1)d}{2}}G_t^{\frac{d}{2}}}\myExp{-\frac{1}{2}\sum^d_{j=1}(A^{-1}_{t}y^{(j)},y^{(j)})}\differential{y},
\end{align*}
where $y^{(j)}=(y^{(j)}_i)_{i\in S}$ is the vector of the $j$-th coordinates of $y_i$, $i\in S$. As $\ve\to0$, this converges to the value of the density of $(\Delta\eta(t_i))_{i\in S}$ at the point $(-q(t_i,t_{i+1})a)_{i\in S}$, namely
$$
\frac{1}{(2\pi)^{\frac{p(m-1)d}{2}}G_t^{\frac{d}{2}}}\myExp{-\frac{\norm{a}^2}{2}(A^{-1}_{t}q_t,q_t)},\quad q_t=(q(t_i,t_{i+1}))_{i\in S},
$$
and, for all $\ve>0$, it is bounded by $(2\pi)^{-\frac{p(m-1)d}{2}}G_t^{-\frac{d}{2}}$. Therefore, by the dominated convergence theorem \cite[Theorem 1.21]{Kallenberg2002}, it suffices to check that
\begin{align}
\label{eq:Gt_integrable}
\int_{\Delta^p_m}G_t^{-\frac{d}{2}}\differential{t_1}\cdots\differential{t_{pm}}<\infty.
\end{align}
Let $e=k(1,\cdot)1_{[0,1]}/\norm{k(1,\cdot)}$. Since $P$ is the orthogonal projection onto the span of the unit vector $e$, and $\norm{k(1,\cdot)}^2=\Var(u_1(1))$,
\begin{align*}
G_t=G\left(\Delta g(t_i),\ i\in S;\ e\right)=\frac{G\left(\Delta u_1(t_i),\ i\in S;\ u_1(1)\right)}{\Var(u_1(1))},
\end{align*}
where $\Delta u_1(t_i)=u_1(t_{i+1})-u_1(t_i)$. Let $s_1\leqslant s_2\leqslant\ldots\leqslant s_{pm}$ be the points $t_1,\ldots,t_{pm}$ arranged in increasing order (the order statistics, in statistical terminology), and put $s_0=0$, $s_{pm+1}=1$. The random variables $\Delta u_1(t_i)$, $i\in S$, and $u_1(1)$ are integer linear combinations of $u_1(s_1),\ldots,u_1(s_{pm+1})$. For almost every $t$, the points $s_1<\ldots<s_{pm}<1$ are distinct, and then these $p(m-1)+1$ combinations are linearly independent, because the increments within each block are differences of values at distinct points, different blocks involve disjoint sets of points, and $u_1(1)$ involves the point $1$. 
Therefore, by  Lemma~\ref{lem:Gram},
\begin{align}
\label{eq:Gt_lower}
G_t\geqslant c\prod_{j=1}^{pm+1}(s_j-s_{j-1})^\zeta
\end{align}
for almost every $t$, where $c>0$ does not depend on $t$; in particular, $G_t>0$ for almost every $t$. The right-hand side of \eqref{eq:Gt_lower} is a symmetric function of $t_1,\ldots,t_{pm}$ and $\Delta^p_m\subset[0,1]^{pm}$. Hence, the desired integrability follows since  
\begin{align*}
\int_{\Delta^p_m}G_t^{-\frac{d}{2}}\differential{t_1}\differential{t_2}\cdots\differential{t_{pm}}&\leqslant c\,(pm)!\int_{0\leqslant s_1\leqslant\ldots\leqslant s_{pm}\leqslant1}\ \prod_{j=1}^{pm+1}(s_j-s_{j-1})^{-\frac{\zeta d}{2}}\differential{s_1}\cdots\differential{s_{pm}}
\\
&
=c\,(pm)!\,\frac{\Gamma(1-\frac{\zeta d}{2})^{pm+1}}{\Gamma((pm+1)(1-\frac{\zeta d}{2}))},
\end{align*}
by the Dirichlet integral formula (see, e.g., \cite[Section 12.5]{WhittakerWatson1996}), which is finite because $\frac{\zeta d}{2}<1$. This proves  \eqref{eq:Gt_integrable}.
 \end{proof}

\begin{remark}
\label{rem:unconditional}
The proof of Theorem \ref{thm:SILTVGB} applies verbatim to the process $u$ itself, \emph{i.e.}, with $P$ replaced by the zero operator and $\eta^a$ by $u$. In this case,  $G_t=G(\Delta u_1(t_i),\ i\in S)$, Lemma \ref{lem:Gram} gives $G_t\geqslant c\prod_{j=1}^{pm}(s_j-s_{j-1})^\zeta$, and the resulting Dirichlet integral is finite. Hence, if $u$ is $(2,\zeta)$-locally nondeterministic with $\zeta\in(0,\frac{2}{d})$, then $T^u_m=\lplim{\ve\to0}T^u_{\ve,m}$ exists for every $p\geqslant1$. This agrees with \cite[Theorem 2.2]{Izyumtseva2026StochascticFlowInteraction} with $\rho\equiv1$. Note that the Lipschitz condition imposed on the kernel there is used only to guarantee the continuity of the sample paths and plays no role in the convergence argument.
\end{remark}

\begin{remark}
  \label{rem:geometric_interpretation} The integrability requirement in \eqref{eq:Gt_integrable} has a direct geometric interpretation. Denote by $\theta_t \in [0, \pi/2]$ the angle between $e = k(1,\cdot)1_{[0,1]}/\norm{k(1,\cdot)}$ and the linear span generated by $\{\Delta g(t_i) : \ i\in S\}$. Then,  
  \begin{align*}
G_t=G\left(\Delta g(t_i),\ i\in S;\ e\right)= \left(\vol{ \Delta g(t_i),\ i\in S} \sin(\theta_t)\right)^2, \quad t\in \Delta_m^p. 
\end{align*}
Here, the notation $\vol{h_1, h_2, \ldots, h_k}$ denotes the volume of the parallelotope spanned by the elements $h_1, h_2, \ldots, h_k$ in a Hilbert space $\mbH$. 
Therefore, \eqref{eq:Gt_integrable} is equivalent to the following geometric requirement: 
\begin{align*}
 \int_{\Delta^p_m} \left(\vol{ \Delta g(t_i),\ i\in S} \sin(\theta_t)\right)^{-{d}}\differential{t_1}\cdots\differential{t_{pm}}<\infty.\end{align*}
\end{remark}

\begin{lemma}
\label{lem:conditional}
Assume that $u$ is $(2,\zeta)$-locally nondeterministic for some $\zeta\in(0,\frac{2}{d})$, and let $p\in\mbN$. Then, the mapping $a \mapsto \Eof{(T^{\eta^a}_m)^p}$ is a version of the conditional expectation $\Eof{(T^u_m)^p\mid u(1)=a}$. 
\end{lemma}
\begin{proof}[Proof of Lemma~\ref{lem:conditional}]
For $\ve>0$, we have $T^u_{\ve,m}=\Phi_\ve(u)$ and $T^{\eta^a}_{\ve,m}=\Phi_\ve(\eta^a)$ for a continuous functional $\Phi_\ve$ on $C([0,1],\mbR^d)$. Since $u=\eta+\vf\,u(1)$ with $\eta$ independent of $u(1)$ (see Remark~\ref{rem:regular_conditional}),
$$
\Eof{(T^u_{\ve,m})^p\mid u(1)=a}=\Eof{\Phi_\ve(\eta+\vf\,a)^p}=\Eof{(T^{\eta^a}_{\ve,m})^p}, \quad\text{for every }a\in\mbR^d.
$$
As $\ve\to0$, $T^u_{\ve,m}\to T^u_m$ in $L^{\max(p,2)}(\Omega, \cF, \prob)$ by Remark \ref{rem:unconditional}, and hence, by H\"older's inequality, $(T^u_{\ve,m})^p\to(T^u_m)^p$ in $L^1(\Omega, \cF, \prob)$. Therefore, $\Eof{(T^u_{\ve,m})^p\mid u(1)}\to\Eof{(T^u_m)^p\mid u(1)}$ in $L^1(\Omega, \cF, \prob)$ and, along a sequence $\ve_n\to0$, almost surely. On the other hand, $\Eof{(T^{\eta^a}_{\ve,m})^p}\to\Eof{(T^{\eta^a}_m)^p}$ for every $a$ by Theorem \ref{thm:SILTVGB}. Thus, $\Eof{(T^u_m)^p\mid u(1)=a}=\Eof{(T^{\eta^a}_m)^p}$ for almost every $a$ with respect to the law of $u(1)$.
\end{proof}
In what follows, $\Eof{(T^u_m)^p\mid u(1)=a}$ will always denote this version. 

\begin{remark}
  Note that self-intersection local times of $\mbR^d$-valued Volterra Gaussian bridges can also be shown to exist almost surely. One possible way to prove this assertion is to use the local time approach of \cite{xiao2002local}. The standard modus operandi is to first consider the occupation measure of an appropriate Gaussian random field constructed by the increments of the Volterra Gaussian bridge, and then prove the existence of the local times as a Radon--Nikodym derivative of the occupation measure with respect to the Lebesgue measure. The second step can be achieved via standard harmonic analytic tools, \emph{e.g.}, by showing  integrability of the Fourier--Stieltjes transform of the occupation measure (see \cite[Theorem 21.9]{GemanHorowitz1980}).  In fact, it can also be shown that the local time thus obtained is jointly continuous in both the time as well as the spatial variables. We omit the proofs of these assertions for  two reasons. First, they can be obtained by using arguments similar to those in \cite{xiao2002local}. Second, our primary focus in this paper is the asymptotic behaviour of the moments of the self-intersection local times of the Volterra Gaussian bridges as the end-to-end distance becomes large. Because our main object is a Gaussian process, the $L^p$ techniques used in this paper are naturally well-suited for this pursuit. We refer the readers to \cite{GemanHorowitz1980} for a comprehensive account of local times of random fields, and the monograph \cite{Marcus2006MarkovProcesses} for an overview of techniques applied to other classes of stochastic processes such as Markov processes.
\end{remark}

\section{Rosen renormalisation}
\label{sec:RR}
In this section, we study self-intersection local times for planar ($d=2$) strongly locally nondeterministic Volterra Gaussian bridges. We prove that the increments of these bridges behave like the increments of the Wiener process and construct Rosen renormalised self-intersection local times \cite{rosen1986renormalized}. 

Let $u$ be a Volterra Gaussian process generated by the kernel $k$
and let $\eta^a$ be the corresponding Volterra Gaussian bridge defined in \eqref{eq:VGB}. In this section and in  Section~\ref{sec:asymptotics2}, we will assume the kernel depends only on the difference of its arguments:
\begin{align}
\label{eq:stationary_kernel}
k(t,s)=k(t-s)1_{[0,t]}(s),\ t,s\in[0,1],
\end{align}
where, with a slight abuse of notation, $k\in C^1([0,1])$, the space of continuously differentiable functions on $[0, 1]$, and $k(0)\neq0$. Such kernels satisfy Assumption~\ref{assumption:kernel} of Section \ref{sec:SILTVGB} with $\vk=1$. Moreover, since $k(0)\neq0$, there is $c>0$ with $$\int_0^hk^2(x)\differential{x}\geqslant ch$$ for $h\in(0,1]$ (see the proof of Lemma \ref{lem:increments} below), so that $u$ is $(2,1)$-locally nondeterministic. 

We write $g(t)=k(t-\cdot)1_{[0,t]}$ and $\Delta g(s,t)=g(t)-g(s)$ for $0\leqslant s<t\leqslant 1$. Note that $k(1,\cdot)=k(1-\cdot)$ on $[0,1]$. The following technical lemma provides a number of useful estimates and asymptotic properties. 

\begin{lemma}\label{lem:increments}Let the kernel $k$ be as in \eqref{eq:stationary_kernel}. Then, the following statements hold.
 \begin{enumerate}
 \item There exist $C>0$ and $\delta\in(0,1]$ such that $\norm{\Delta g(s,t)}^2\leqslant C(t-s)$ for all $0\leqslant s<t\leqslant 1$ and
 \begin{align*}
\frac{1}{2}k^2(0)(t-s)\leqslant\norm{(I-P)\Delta g(s,t)}^2\leqslant C(t-s), \quad\text{whenever}\ t-s\leqslant\delta.
 \end{align*}  
 Moreover, $\norm{(I-P)\Delta g(s,t)}^2\sim k^2(0)(t-s)$ as $t-s\to0$, uniformly in $s$.
 \item For every $h\in L^2([0,1])$,
 $
\left(\frac{\Delta g(s,t)}{\norm{(I-P)\Delta g(s,t)}},h\right)\to0\quad\text{as}\ t-s\to0,\ \text{uniformly in}\ s,
 $
 \emph{i.e.}, the normalised increments $\Delta g(s,t)/\norm{(I-P)\Delta g(s,t)}$ converge weakly to $0$ in $L^2([0,1])$ as $t-s\to0$.
 \item There exists $C>0$ such that for all $0\leqslant s_1<s_2\leqslant s_3<s_4\leqslant 1$, 
 \begin{gather*}
|(\Delta g(s_1,s_2),\Delta g(s_3,s_4))|\leqslant C(s_2-s_1)(s_4-s_3),\quad \\
 \text{ and }\quad 
|(\Delta g(s_1,s_2),k(1,\cdot))|\leqslant C(s_2-s_1).
 \end{gather*}
 In particular, the increments over disjoint intervals are asymptotically orthogonal:
 $$
 \frac{|(\Delta g(s_1,s_2),\Delta g(s_3,s_4))|}{\norm{\Delta g(s_1,s_2)}\,\norm{\Delta g(s_3,s_4)}}\leqslant C'\min(s_2-s_1,\,s_4-s_3)^{\frac{1}{2}}.
 $$
 \end{enumerate}
\end{lemma}
\begin{proof}[Proof of Lemma~\ref{lem:increments}] 
We will prove the last assertion first. For $r\leqslant s_3$, both indicators in $\Delta g(s_3,s_4)(r)=k(s_4-r)-k(s_3-r)$ are equal to $1$. Hence, $$|\Delta g(s_3,s_4)(r)|\leqslant\norm{k'}_\infty(s_4-s_3)$$ for $r\leqslant s_3$. Moreover, $\Delta g(s_1,s_2)$ vanishes outside $[0,s_2]\subset[0,s_3]$, and
\begin{align*}
\int_0^{s_2}|\Delta g(s_1,s_2)(r)|\differential{r}&\leqslant\int_{s_1}^{s_2}|k(s_2-r)|\differential{r}+\int_0^{s_1}|k(s_2-r)-k(s_1-r)|\differential{r}\\
&\leqslant(\norm{k}_\infty+\norm{k'}_\infty)(s_2-s_1).
\end{align*}
These two estimates give both inequalities in (3). 

Since the operator $I-P$ is an orthogonal projection,
\begin{align*}
\norm{(I-P)\Delta g(s,t)}^2&{}=\norm{\Delta g(s,t)}^2-\norm{P\Delta g(s,t)}^2\\
&{}=\int^{t-s}_{0}k^2(x)\differential{x}+\int^{s}_0(k(t-r)-k(s-r))^2\differential{r}-\frac{(\Delta g(s,t),k(1,\cdot))^2}{\norm{k(1,\cdot)}^2}.
\end{align*}
The first summand equals $k^2(0)(t-s)+o(t-s)$ as $t-s\to0$, uniformly in $s$, because $k$ is continuous at $0$. The second summand is at most $\norm{k'}_\infty^2(t-s)^2$. Finally,  the third one is at most $C^2(t-s)^2/\norm{k(1,\cdot)}^2$. This gives the asymptotic relation, the bound $\norm{\Delta g(s,t)}^2\leqslant(\norm{k}^2_\infty+\norm{k'}^2_\infty)(t-s)$ and, for $\delta$ small enough, the two-sided bound for $t-s\leqslant\delta$. Choosing $\delta$ smaller if necessary, we also have $\int_0^{h}k^2(x)\differential{x}\geqslant\frac34k^2(0)\min(h,\delta)\geqslant\frac34k^2(0)\,\delta h$ for $h\in(0,1]$, so that $\norm{\Delta g(s,t)}^2\geqslant\frac34k^2(0)\,\delta\,(t-s)$ for all $s<t$. Together with the first inequality in (3), this proves the last statement of (3).

Now, for $h\in L^2([0,1])$,
$$
|(\Delta g(s,t),h)|\leqslant\norm{k}_\infty(t-s)^{\frac{1}{2}}\left(\int_s^th^2(r)\differential{r}\right)^{\frac{1}{2}}+\norm{k'}_\infty(t-s)\norm{h}.
$$
Dividing by $\norm{(I-P)\Delta g(s,t)}\geqslant|k(0)|(t-s)^{\frac{1}{2}}/\sqrt{2}$ (for $t-s\leqslant\delta$), we obtain a bound that tends to $0$ as $t-s\to0$, uniformly in $s$, by the absolute continuity of the Lebesgue integral.
\end{proof}

\begin{lemma}
\label{lem:slnd}
 Let the kernel $k$ be as in \eqref{eq:stationary_kernel}, and let $$\tilde{g}(t)=(I-P)k(t-\cdot)1_{[0,t]},\quad t\in[0,1].$$ Then, for every $\delta_0\in(0,1)$, the function $\tilde{g}$ is strongly locally nondeterministic on the configurations $(t_1,\ldots,t_m)\in\Delta_m$ with $t_m-t_1\leqslant1-\delta_0$.
 \end{lemma}
\begin{proof}[Proof of Lemma~\ref{lem:slnd}]
Let $M\subseteq\{1,\ldots,m-1\}$ and write $e_i=\Delta\tilde{g}(t_i)/\norm{\Delta\tilde{g}(t_i)}$ for $i\in M$. Then,
$$
\frac{G(\Delta\tilde{g}(t_1),\ldots,\Delta\tilde{g}(t_{m-1}))}{G(\Delta\tilde{g}(t_i),\ i\notin M)\prod_{i\in M}\norm{\Delta\tilde{g}(t_i)}^2}=G(Qe_i,\ i\in M),
$$
where $Q$ is the orthogonal projection onto the orthogonal complement of the span of the vectors $\Delta\tilde{g}(t_i)$, $i\notin M$. If $t_1<\ldots<t_m$ and $t_m-t_1\leqslant1-\delta_0$, then $t_1>0$ or $t_m<1$, and in either case the increments $\Delta\tilde{g}(t_i)$, $i=1,\ldots,m-1$, are linearly independent (because $k(0)\neq0$, the functions $g(s)$, $s\in(0,1]$, at distinct points are linearly independent). Since $((I-P)x,(I-P)y)=(x,y)-(x,e)(y,e)$ with $e=k(1,\cdot)/\norm{k(1,\cdot)}$, Lemma \ref{lem:increments}(1),(3) shows that the vectors $e_i$, $i\in M$, become asymptotically orthonormal as $\max_{i\in M}\Delta t_i\to0$, and parts (2) and (3) of Lemma \ref{lem:increments} show that they become asymptotically orthogonal to the increments over the remaining intervals. Then,  arguments similar to those used in  \cite{Dorogovtsev2011regularization} then give $G(Qe_i,\ i\in M)\to1$.
\end{proof}
\begin{remark}
\label{rem:slnd_restriction}
Note that the restriction $t_m-t_1\leqslant1-\delta_0$ cannot be dropped. Since $\tilde{g}(0)=\tilde{g}(1)=0$, the increments of $\tilde{g}$ over a partition of $[0,1]$ sum to zero. For instance, if $m=3$, $t_1=0$, $t_3=1$ and $M=\{1\}$, then $\Delta\tilde{g}(t_1)+\Delta\tilde{g}(t_2)=0$, so the left-hand side in Definition \ref{def:strongLND} vanishes, while the right-hand side equals $\norm{\Delta\tilde{g}(t_2)}^2\norm{\Delta\tilde{g}(t_1)}^2>0$ for every $t_2\in(0,1)$.
\end{remark}

In order to obtain the Rosen renormalisation, we need the notion of the Fourier--Wiener transform (also known as the $S$-transform in white noise analysis, see \cite{HidaKuoPotthoffStreit1993} or \cite[Definition 5.8, p.39]{Kuo1996WhiteNoise}). Consider two independent white noises $\xi_1,\ \xi_2$ in $L^2([0,1])$ generated by two independent Wiener processes $w_1,\ w_2.$ Let us first define the stochastic exponential \cite[Chapter 3]{Janson1997GHS}. 

  \begin{definition}[Stochastic exponential]\label{def:stochastic_exponent}
    For two elements $h_1, h_2 \in L^2([0,1])$, the stochastic exponential is defined as 
    \begin{align*}
      \mathcal{E}(h_1, h_2) \coloneq  \exp\left((h_1, \xi_1) + (h_2, \xi_2) - \frac{1}{2} \left(\norm{h_1}^2 + \norm{h_2}^2 \right)\right).
    \end{align*}
  \end{definition}

  Let $\theta$ be a square-integrable random variable that is measurable with respect to $\xi_1,\ \xi_2$. Then, its Fourier--Wiener transform is defined as follows.
  \begin{definition}[Fourier--Wiener transform]\label{def:fourier_wiener}
    The Fourier--Wiener transform of $\theta$ is defined as 
    \[
      \mathcal{T}(\theta)(h_1, h_2) \coloneq  \Eof{\theta \mathcal{E}(h_1, h_2)},\qquad h_1,h_2\in L^2([0,1]).
    \]
  \end{definition}
 Considering the Fourier--Wiener transform of the approximations $T^{\eta}_{\ve,m}$ and formally passing to the limit, we conclude that the formal Fourier--Wiener transform of the random variable 
$T_{m}^{\eta}$ has the form
\begin{multline}
\label{eq4.3}
\int_{\Delta_m}\frac{1}{(2\pi)^{m-1}G((I-P)\Delta g(t_1),\ldots,(I-P)\Delta g(t_{m-1}))}
\\
\times
\myExp{-\frac{1}{2}(\|\hat{P}_{t_1\ldots t_m}h_1\|^2+\|\hat{P}_{t_1\ldots t_m}h_2\|^2)}\,\differential{t_1}\,\differential{t_2}\cdots\differential{t_m},
\end{multline}
where $\hat{P}_{t_1\ldots t_m}$ is the orthogonal projection onto the linear span of the vectors $(I-P)\Delta g(t_i)$, $i=1,\ldots,(m-1)$. Unfortunately, the integral \eqref{eq4.3} diverges. Therefore, we need to construct a regularisation of this divergent integral, which we do next. Let $\tilde{\Delta} g(t_i)$, $i=1,\ldots,m-1$, be the orthogonal system obtained from the vectors $(I-P)\Delta g(t_i)$, $i=1,\ldots,m-1$, by the Gram--Schmidt orthogonalisation procedure, and let $P_{\tilde{\Delta} g(t_i)}$ be the orthogonal projection onto the linear span generated by $\tilde{\Delta} g(t_i)$. Since
$$
G((I-P)\Delta g(t_1),\ldots,(I-P)\Delta g(t_{m-1}))=\prod^{m-1}_{i=1}\|\tilde{\Delta} g(t_i)\|^2
$$
and $$\hat{P}_{t_1\ldots t_m}=\sum^{m-1}_{i=1}P_{\tilde{\Delta} g(t_i)},$$ the integrand in \eqref{eq4.3} factorises over $i=1,\ldots,m-1$. The following theorem describes the Fourier--Wiener transform of the Rosen renormalised $m$-multiple self-intersection local time of the planar Volterra Gaussian bridge $\eta$.
\begin{theorem}
\label{thm:Fourier_Wiener} Let the kernel $k$ be as in \eqref{eq:stationary_kernel}. Then, the Rosen renormalised self-intersection local time of multiplicity $m$ for the Volterra Gaussian bridge $\eta$ exists for any $m\geqslant 2$, \emph{i.e.}, the integral 
\begin{multline*}
\int_{\Delta_m}\frac{1}{(2\pi)^{m-1}}\prod^{m-1}_{i=1}\frac{1}{\|\tilde{\Delta} g(t_i)\|^2}
\left(\myExp{-\frac{1}{2}(\|P_{\tilde{\Delta} g(t_i)} h_1\|^2+\|P_{\tilde{\Delta} g(t_i)} h_2\|^2)}-1\right)\differential{t_1}\cdots\differential{t_m}
\end{multline*}
converges for any $h_1,h_2\in L^2([0,1])$. 
\end{theorem}
We omit the proof of Theorem~\ref{thm:Fourier_Wiener} since it can be obtained 
by applying Lemma~\ref{lem:increments} and Lemma~\ref{lem:slnd} and using similar arguments as those used in the proof of \cite[Theorem 4.1]{Izyumtseva2026VGP}. Note that  Theorem~\ref{thm:Fourier_Wiener} extends the well-known Varadhan--Rosen renormalisation (see \cite{Varadhan1969}, \cite{rosen1986renormalized}, \cite{LeGall1985}) of self-intersection local times of the planar Wiener process to planar strongly locally nondeterministic Volterra Gaussian bridges conditioned to be $0$ at time $1$. 
\section{Asymptotics of the conditional moments under $(2,\zeta)$-local nondeterminism}
 \label{sec:ACM} 
Let $\{u(t)$ : $t\in[0,1]\}$ be a $(2,\zeta)$-locally nondeterministic Volterra 
Gaussian process in $\mathbb{R}^d$ with $\zeta\in(0,\frac{2}{d})$, and consider the formal expression 
\[
T^u_m = \int_{\Delta_m}\prod_{i=1}^{m-1}
\delta_0(u(t_{i+1})-u(t_i))\differential{t_1}\cdots\differential{t_m}.
\]
In addition to the requirement in \eqref{eq:LND}, we assume that there exist constants $c_0,C_0>0$ such that,  for all $0\leqslant s<t\leqslant 1,$
\begin{align}
\left|\int_0^1(k(t,r)-k(s,r))k(1,r)\differential{r}\right|&\geqslant c_0(t-s),\label{eq:drift}\\
\int_0^1(k(t,r)1_{[0,t]}-k(s,r)1_{[0,s]})^2\differential{r}&\leqslant C_0(t-s)^\zeta.\label{eq:upper}
\end{align}
Recall that $k(t,r)=0$ for $r>t$ since $k$ is a Volterra kernel. In probabilistic terms, the inequality in \eqref{eq:drift} means that $|\Cov(u_1(t)-u_1(s),u_1(1))|\geqslant c_0(t-s)$, \emph{i.e.}, the mean $q(s,t)a$ of the increment $\eta^a(t)-\eta^a(s)$ grows at least linearly in $t-s$. On the other hand, the inequality in \eqref{eq:upper} means that $\Eof{(u_1(t)-u_1(s))^2}\leqslant C_0(t-s)^\zeta$, which together with \eqref{eq:LND} gives $\Eof{(u_1(t)-u_1(s))^2}\asymp(t-s)^\zeta$.

\begin{remark}
\label{rem:fBm}
Conditions \eqref{eq:drift} and \eqref{eq:upper} hold for the fractional Brownian motion $B^H$ with Hurst index $H\in(0,1)$, represented as a Volterra Gaussian process via the Molchan--Golosov kernel \cite{MolchanGolosov1969}, \cite{NorrosValkeilaVirtamo1999}. It is, in fact,  $(2,2H)$-locally nondeterministic. The Wiener process of the Molchan--Golosov representation generates the same filtration as $B^H$ \cite[Theorem 3.2]{NorrosValkeilaVirtamo1999}, so that $$\Var(B^H(t)\mid\cF^w_s)=\Var(B^H(t)\mid \sigma(B^H(r),\ r\leqslant s)) \geqslant c(t-s)^{2H}$$ for $0\le s < t\le 1$,  by the strong local nondeterminism of $B^H$ \cite[Lemma 7.1]{Pitt1978}. Indeed, $\Eof{(B^H(t)-B^H(s))^2}=(t-s)^{2H}$, and $\Cov(B^H(t)-B^H(s),B^H(1))=\frac{1}{2}(f(t)-f(s))$ with $f(x)=x^{2H}-(1-x)^{2H}$. Now, since $f'(x)=2H(x^{2H-1}+(1-x)^{2H-1})\geqslant2H$ on $(0,1)$, the condition \eqref{eq:drift} holds with $c_0=H$. On the other hand, \eqref{eq:drift} fails if $t\mapsto\Cov(u_1(t),u_1(1))$ has a critical point in $[0,1]$. For example, for $k(t,s)=(1-(t-s))1_{[0,t]}(s)$, one has $\Cov(u_1(t),u_1(1))=\frac{t^2}{2}-\frac{t^3}{6}$, whose derivative vanishes at $t=0$.
\end{remark}

 To describe the asymptotics of the conditional moments $\Eof{(T^u_m)^p\mid u(1)=a}$ for  $p\in\mbN$ as $\norm{a}\to\infty$, we need the following technical lemmas, which are of independent interest. 
 
 \subsection{Some useful technical lemmas}
 Recall that,  for $m\geqslant 2$, 
 \begin{align*}
\Delta_m &=\{(t_1,\dots,t_m)\in\mathbb{R}^m:\ 0\le t_1\le\dots\le t_m\le1\},
\end{align*}
and $\Delta_m^p =\Delta_m\times\cdots\times\Delta_m $ ($p$ factors).
A point
$t=(t_1,\dots,t_{pm})\in\Delta_m^p$ consists of the $p$ blocks
$(t_{(b-1)m+1},\dots,t_{bm})$, $b=1,\dots,p$, and
$
S=\{1,\dots,pm-1\}\setminus\{m,2m,\dots,(p-1)m\}
$
is the set of indices of the increments within the blocks.
For $t\in\Delta_m^p$,  let $s_1\le\dots\le s_{pm}$ be the order statistics of
$t_1,\dots,t_{pm}$. Put $s_0=0$, $s_{pm+1}=1$.

\begin{lemma}\label{lem:gaps}
Let $p\in\mbN$, $m\ge2$, $\tau\in(0,1)$ and $\chi>0$. Further, let 
$$
F_\tau(t)=\prod_{j=1}^{pm+1}(s_j-s_{j-1})^{-\tau}\in(0,\infty],
$$
with the convention $0^{-\tau}=\infty$.
Then, there exists $C<\infty$, depending only on $p$, $m$, $\tau$ and $\chi$, such that for all $\varepsilon\in(0,1]$
\begin{equation}\label{eq:gaps}
\int_{\Delta_m^p}F_\tau(t)\exp\Bigl(-\sum_{i\in S}\Big(\frac{t_{i+1}-t_i}{\varepsilon}\Big)^{\chi}\Bigr)\,\differential{t}
\le C\,\varepsilon^{p(m-1)(1-\tau)}.
\end{equation}
\end{lemma}

\begin{proof}[Proof of Lemma~\ref{lem:gaps}]
  First, note that outside a Lebesgue null set, the points $t_1,\dots,t_{pm}$ are distinct and lie
in the open interval $(0,1)$. Hence, up to a null set, $\Delta_m^p$ is the disjoint union of finitely
many open cells $O_\sigma$. Here $\sigma$ runs over the relative orders of
$t_1,\dots,t_{pm}$ that are compatible with $t_{(b-1)m+1}<\dots<t_{bm}$,
$b=1,\dots,p$. Fix such a $\sigma$. For $t\in O_\sigma$,  let $\pi(l)$ be the
position of $t_l$ in the increasing rearrangement, so that $t_l=s_{\pi(l)}$, and put
$g_j=s_j-s_{j-1}$, $j=1,\dots,(pm+1)$. The map $\pi$ depends only on $\sigma$. For
$i\in S$,  we have $\pi(i)<\pi(i+1)$ and
\begin{equation}\label{eq:incr-gaps}
t_{i+1}-t_i=\sum_{j\in J_i}g_j,\qquad J_i=\{\pi(i)+1,\dots,\pi(i+1)\}\subseteq\{2,\dots,pm\}.
\end{equation}
For $i\in S$,  put $\iota(i)=\pi(i+1)$. Since $\pi(i)<\pi(i+1)$, we have
$\iota(i)\in J_i$, and $\iota\colon S\to\{2,\dots,pm\}$ is injective because the map
$\pi$ is. Let $R=\iota(S)$, so that $|R|=p(m-1)$, and let
$N=\{1,\dots,pm+1\}\setminus R$. Then, $|N|=p+1$ and $1,pm+1\in N$. By
\eqref{eq:incr-gaps}, $t_{i+1}-t_i\ge g_{\iota(i)}$ for $i\in S$.

The map $t\mapsto(g_1,\dots,g_{pm})$ is a linear bijection of $O_\sigma$ onto the
open simplex
\[
\Sigma=\{(g_1,\dots,g_{pm}):\ g_j>0,\ g_1+\dots+g_{pm}<1\}.
\]
Its Jacobian has modulus one, and $g_{pm+1}=1-\sum_{j\le pm}g_j$. Therefore, we have 
\[
\int_{O_\sigma}F_\tau(t)\,e^{-\sum_{i\in S}((t_{i+1}-t_i)/\varepsilon)^{\chi}}\,\differential{t}
\le\int_\Sigma\prod_{j=1}^{pm+1}g_j^{-\tau}\prod_{j\in R}e^{-(g_j/\varepsilon)^{\chi}}\,\differential{g_1}\cdots \differential{g_{pm}} \eqqcolon K_\sigma(\varepsilon).
\]
Fix $(g_j)_{j\in R}$ with $\varrho\coloneq1-\sum_{j\in R}g_j>0$, and integrate over the
remaining $p$ free variables $g_j$, $j\in N\setminus\{pm+1\}$, subject to
$\sum_{j\in N\setminus\{pm+1\}}g_j<\varrho$. By the Dirichlet integral formula (\emph{e.g.}, see \cite[Section 12.5]{WhittakerWatson1996}), this
integral of $\prod_{j\in N}g_j^{-\tau}$ equals
\[
D\,\varrho^{e},\qquad e=(p+1)(1-\tau)-1,\qquad D=\frac{\Gamma(1-\tau)^{p+1}}{\Gamma((p+1)(1-\tau))}.
\]
Note that $e>-1$, but $e<0$ when $\tau>p/(p+1)$, and then $\varrho^{e}$ is unbounded
as $\sum_{j\in R}g_j\to1$. We therefore split the remaining integration over
$(g_j)_{j\in R}$ according to whether $\sum_{j\in R}g_j\le\frac12$ or
$\sum_{j\in R}g_j>\frac12$.

Let us first consider the set $\{\sum_{j\in R}g_j\le\frac12\}$. We have $\varrho\in[\frac12,1]$, and
hence,  $\varrho^{e}\le\max(1,2^{-e})\le2$ because $e>-1$. Hence, this part
of $K_\sigma(\varepsilon)$ is at most
\[
2D\prod_{j\in R}\int_0^\infty g^{-\tau}e^{-(g/\varepsilon)^{\chi}}\,\differential{g}
=2D\Big(\frac1\chi\,\Gamma\Big(\frac{1-\tau}{\chi}\Big)\Big)^{p(m-1)}\varepsilon^{p(m-1)(1-\tau)}.
\]
On the complementary set $\{\sum_{j\in R}g_j>\frac12\}$, some $g_j$ with $j\in R$
exceeds $1/(2p(m-1))$, since $|R|=p(m-1)$. Hence, 
$\prod_{j\in R}\myExp{-(g_j/\varepsilon)^{\chi}}\le\exp(-(2p(m-1)\varepsilon)^{-\chi})$,
and this part of $K_\sigma(\varepsilon)$ is at most
\[
\exp\bigl(-(2p(m-1)\varepsilon)^{-\chi}\bigr)\int_\Sigma\prod_{j=1}^{pm+1}g_j^{-\tau}\,\differential{g}.
\]
The last integral is a finite Dirichlet integral, and
$$\exp(-(2p(m-1)\varepsilon)^{-\chi})\le C'\varepsilon^{p(m-1)(1-\tau)}$$
for $\varepsilon\in(0,1]$. Thus,  $K_\sigma(\varepsilon)\le C\varepsilon^{p(m-1)(1-\tau)}$, and summing over 
finitely many $\sigma$ proves the inequality in \eqref{eq:gaps}.
\end{proof}

\begin{lemma}\label{lem:block}
Assume that $u$ is $(2,\zeta)$-locally nondeterministic and let $c_\zeta$ be as in
Lemma~\ref{lem:Gram}. Assume also the inequality in \eqref{eq:upper}. Let $0\le t_1<t_2<\dots<t_m\le1$, and let $A$ and $B$ be the covariance matrices of $(\eta_1(t_{i+1})-\eta_1(t_i))_{i=1}^{m-1}$ and $(X_i)_{i=1}^{m-1}$ with $X_i=u_1(t_{i+1})-u_1(t_i)$, respectively. Write $\nu=\|k(1,\cdot)\|^2=\Var(u_1(1))$, and  define $q=(q(t_i,t_{i+1}))_{i=1}^{m-1}$ as follows 
$$
q(s,t)\coloneq\vf(t)-\vf(s)=\frac{\int^{t}_0k(1,r)k(t,r)\differential{r}-\int^{s}_0k(1,r)k(s,r)\differential{r}}{\norm{k(1,\cdot)}^2}.
$$
Then,  the following statements hold.
\begin{enumerate}[label=\textup{(\roman*)},leftmargin=*,itemsep=2pt]
\item For $i=1,\dots,m-1$,  
$c_\zeta(t_{i+1}-t_i)^\zeta\le B_{ii}\le C_0(t_{i+1}-t_i)^\zeta. $
\item For all $x\in\mbR^{m-1}$,
\[
(B^{-1}x,x)\le\rho_m^{-1}\sum_{i=1}^{m-1}\frac{x_i^2}{B_{ii}},\qquad \rho_m=\Big(\frac{c_\zeta}{C_0}\Big)^{m-1}(m-1)^{-(m-2)}.
\]
\item $A=B-\nu\,qq^{\mathsf T}$. If $A$ is nonsingular, then $\nu(B^{-1}q,q)<1$ and
$
(A^{-1}q,q)=\frac{(B^{-1}q,q)}{1-\nu(B^{-1}q,q)}.
$
\end{enumerate}
\end{lemma}

\begin{proof}[Proof of Lemma~\ref{lem:block}]
(i) The upper bound follows from \eqref{eq:upper}. Since the random variables $X_1,\dots,X_{i-1}$ are
$\mathcal F^{w_1}_{t_i}$-measurable,
\begin{align}
\begin{aligned}
B_{ii}{}=\Eof{(u_1(t_{i+1})-u_1(t_i))^2}
{}&\geqslant\Var(X_i\mid X_1,\dots,X_{i-1})\\
{}&\ge\Var(X_i\mid\mathcal F^{w_1}_{t_i})\\
{}&=\int_{t_i}^{t_{i+1}}k^2(t_{i+1},r)\,\differential{r}\ge c_\zeta(t_{i+1}-t_i)^\zeta. 
\end{aligned}
\label{eq:condvar}
\end{align}

(ii) Let $\Lambda=\operatorname{diag}(B_{11}^{1/2},\dots,B_{m-1,m-1}^{1/2})$ and let
$\mathcal R=\Lambda^{-1}B\Lambda^{-1}$ be the correlation matrix. By
\eqref{eq:condvar}, we have 
\[
\det B=\prod_{i=1}^{m-1}\Var(X_i\mid X_1,\dots,X_{i-1})\ge c_\zeta^{m-1}\prod_{i=1}^{m-1}(t_{i+1}-t_i)^\zeta,
\]
and by (i), $\prod_iB_{ii}\le C_0^{m-1}\prod_i(t_{i+1}-t_i)^\zeta$. Hence, 
$\det\mathcal R\ge(c_\zeta/C_0)^{m-1}$. The eigenvalues of $\mathcal R$ are
positive, and they do not exceed $\operatorname{tr}\mathcal R=m-1$. Consequently, 
the smallest eigenvalue of $\mathcal R$ is at least
$\det\mathcal R/(m-1)^{m-2}\ge\rho_m$, and
\[
(B^{-1}x,x)=(\mathcal R^{-1}\Lambda^{-1}x,\Lambda^{-1}x)\le\rho_m^{-1}|\Lambda^{-1}x|^2=\rho_m^{-1}\sum_{i}x_i^2/B_{ii}.
\]

(iii) By the representation in \eqref{eq:VGB}, $\eta_1(t_{i+1})-\eta_1(t_i)=X_i-q_iu_1(1)$ with $q_i=q(t_i,t_{i+1})$,
and $\Cov(X_i,u_1(1))=(\Delta g(t_i),k(1,\cdot))=\nu q_i$. Hence, 
$A_{ij}=B_{ij}-2\nu q_iq_j+\nu q_iq_j=B_{ij}-\nu q_iq_j$.

By the matrix determinant lemma (see, e.g., \cite[Section 0.8.5]{HornJohnson2013}), $$\det A=(1-\nu(B^{-1}q,q))\det B. $$ Here, 
$\det B>0$ by the proof of (ii), and $\det A>0$ because $A$ is a nonsingular
covariance matrix. Hence, we have  $\nu(B^{-1}q,q)<1$. The formula then follows from the
Sherman--Morrison formula \cite[Section 0.7.4]{HornJohnson2013}
$$A^{-1}=B^{-1}+\nu B^{-1}qq^{\mathsf T}B^{-1}/(1-\nu(B^{-1}q,q)).$$
\end{proof}

We are now ready to state the main result of this section, which gives the asymptotics of the conditional moments of $T^u_m$ as $\|a\|\to\infty$.

\subsection{Main theorem on asymptotics under $(2,\zeta)$-local nondeterminism}

\begin{theorem}\label{thm:4.2}
Let $\zeta\in(0,\frac2d)$, and assume that the $\mbR^d$-valued Volterra Gaussian process
$\{u(t)$ : $t\in[0,1]\}$ is $(2,\zeta)$-locally nondeterministic and
that the inequalities \eqref{eq:drift} and \eqref{eq:upper} hold with this $\zeta$. For $p\in\N$ and $m \ge 2$, put
\begin{equation}\label{eq:gamma}
\gamma_{p,m}=p(m-1)\,\frac{2-\zeta d}{2-\zeta}.
\end{equation}
Then, for all $p\in\N$, and $m \ge 2$, there exist constants $0<c\le C<\infty$ such that
\begin{equation}\label{eq:two-sided}
c\,\|a\|^{-\gamma_{p,m}}\le \E\big((T^u_m)^p\mid u(1)=a\big)\le C\,\|a\|^{-\gamma_{p,m}},
\qquad \|a\|\ge 1.
\end{equation}
In particular,
\[
\lim_{\|a\|\to\infty}\frac{\log \E\big((T^u_m)^p\mid u(1)=a\big)}{\log\|a\|}=-\gamma_{p,m}.
\]
\end{theorem}
\begin{proof}[Proof of Theorem~\ref{thm:4.2}]
Write $$M_p(a)=\E((T^u_m)^p\mid u(1)=a),\quad \nu=\|k(1,\cdot)\|^2, $$ and
$\Delta t_i=t_{i+1}-t_i$. We use the notation $S$, $\Delta g(t_i)$, $q(s,t)$,
$A_t$, $G_t$ and $q_t$ from the proof of Theorem~\ref{thm:SILTVGB}.

By Lemma~\ref{lem:conditional}, $M_p(a)=\E((T^{\eta^a}_m)^p)$. Since
$T^{\eta^a}_{\varepsilon,m}\to T^{\eta^a}_m$ in $L^p(\Omega, \cF, \prob)$, the computation in the proof
of Theorem~\ref{thm:SILTVGB} and the dominated convergence theorem give
\begin{equation}\label{eq:Mp}
M_p(a)=\int_{\Delta_m^p}\frac{1}{(2\pi)^{\frac{p(m-1)d}{2}}G_t^{\frac d2}}
\myExp{-\frac{\|a\|^2}{2}(A_t^{-1}q_t,q_t)}\,\differential{t_1}\cdots \differential{t_{pm}}.
\end{equation}
The integrand is strictly positive for almost every $t$ and continuous in $a$. It is
bounded by $(2\pi)^{-p(m-1)d/2}G_t^{-d/2}$, which is integrable by \eqref{eq:Gt_integrable}. Hence, $M_p$
is continuous, strictly positive and bounded on $\R^d$.

It therefore suffices to prove \eqref{eq:two-sided} for $\|a\|\ge R$ with some
$R\ge1$. Indeed, on $\{1\le\|a\|\le R\}$ both $M_p(a)$ and $\|a\|^{-\gamma_{p,m}}$
are bounded above and below by positive constants.

Fix $t$ with $G_t>0$ and $i\in S$. By the
Cauchy--Schwarz inequality for the inner product $(x,y)\mapsto(A_tx,y)$, applied to
the $i$-th unit vector $e_i$,
\[
q(t_i,t_{i+1})^2=(e_i,q_t)^2\le(A_te_i,e_i)\,(A_t^{-1}q_t,q_t).
\]
Moreover, $$(A_t)_{ii}=\|(I-P)\Delta g(t_i)\|^2\le\|\Delta g(t_i)\|^2\le C_0(\Delta t_i)^\zeta$$
by \eqref{eq:upper}, and $|q(t_i,t_{i+1})|\ge c_0\Delta t_i/\nu$ by \eqref{eq:drift}. Hence,
$(A_t^{-1}q_t,q_t)\ge c_0^2(\Delta t_i)^{2-\zeta}/(C_0\nu^2)$ for each $i\in S$.
Averaging over $i\in S$,
\begin{equation}\label{eq:exponent}
(A_t^{-1}q_t,q_t)\ge c_1\sum_{i\in S}(\Delta t_i)^{2-\zeta},\qquad c_1=\frac{c_0^2}{p(m-1)C_0\nu^2}.
\end{equation}

By the inequality in \eqref{eq:Gt_lower}, $G_t^{-d/2}\le c\,F_{\zeta d/2}(t)$ for almost
every $t$, with $F_\tau$ as in Lemma~\ref{lem:gaps}. Therefore, by \eqref{eq:Mp} and
\eqref{eq:exponent},
\[
M_p(a)\le c\int_{\Delta_m^p}F_{\zeta d/2}(t)\exp\Big(-\frac{c_1\|a\|^2}{2}\sum_{i\in S}(\Delta t_i)^{2-\zeta}\Big)\differential{t}.
\]
Now, we apply Lemma~\ref{lem:gaps} with $\tau=\frac{\zeta d}{2}\in(0,1)$, $\chi=2-\zeta$
and $\varepsilon=(c_1\|a\|^2/2)^{-1/(2-\zeta)}$. Note that $\varepsilon\le1$ for
$\|a\|^2\ge2/c_1$. This yields
\[
M_p(a)\le C\varepsilon^{p(m-1)(1-\frac{\zeta d}{2})}
=C\Big(\frac{c_1}{2}\Big)^{-\gamma_{p,m}/2}\|a\|^{-\gamma_{p,m}}.
\]

By Jensen's inequality,
$M_p(a)=\E ((T^{\eta^a}_m)^p) \ge(\E ( T^{\eta^a}_m))^p=M_1(a)^p$. Since
$\gamma_{p,m}=p\gamma_{1,m}$, it suffices to consider $p=1$. Then,
$S=\{1,\dots,m-1\}$, and $A_t$ is the matrix $A$ of Lemma~\ref{lem:block}. We write
$B_t$ for the corresponding matrix $B$.

Let $\psi(t)=\Cov(u_1(t),u_1(1))=\nu\varphi(t)$, so that $q(s,t)=(\psi(t)-\psi(s))/\nu$.
By \eqref{eq:upper}, $|\psi(t)-\psi(s)|\le\nu^{1/2}(C_0|t-s|^\zeta)^{1/2}$, so $\psi$ is
continuous. Moreover, by \eqref{eq:drift}, $\psi$ is also injective. Hence, $\psi$ is strictly monotone and,
by Lebesgue's differentiation theorem for monotone functions (see \emph{e.g.}, \cite[Theorem 3.23]{Folland1999}; also called monotone differentiation theorem \cite[Theorem 1.6.25]{Tao2011Measure}), has a finite derivative at almost every point of $(0,1)$.
For $M>0$ and $\eta\in(0,\frac12],$ put
\[
\Theta_{M,\eta}=\big\{\theta\in[0,\tfrac12]:\ |\psi(\theta+h)-\psi(\theta)|\le Mh\ \text{ for all } h\in[0,\eta]\big\}.
\]
Each set $\Theta_{M,\eta}$ is closed. Every point of $[0,\frac12)$ at which $\psi$ is
differentiable lies in $\Theta_{M,1/n}$ for some integers $M\ge1$ and $n\ge2$.
Hence, we can fix $M$ and $\eta$ such that $\Theta=\Theta_{M,\eta}$ has positive
Lebesgue measure $|\Theta|$.

For $\varepsilon\in(0,\eta/(m-1)]$,  consider the set 
\[
\mathcal D_\varepsilon=\big\{t\in\Delta_m:\ t_1\in\Theta,\ \Delta t_i\in[\tfrac\varepsilon2,\varepsilon]\ \text{ for } i=1,\dots,m-1\big\},
\qquad |\mathcal D_\varepsilon|=|\Theta|\,(\tfrac\varepsilon2)^{m-1}.
\]
Let $t\in\mathcal D_\varepsilon$. Since $t_m-t_1\le(m-1)\varepsilon\le\eta$ and
$t_1\in\Theta$, for $i=1,\ldots,(m-1),$
\[
\nu\,|q(t_i,t_{i+1})|\le|\psi(t_{i+1})-\psi(t_1)|+|\psi(t_i)-\psi(t_1)|\le2M(m-1)\varepsilon.
\]
Hence, by Lemma~\ref{lem:block}(i),(ii), we immediately have 
\[
(B_t^{-1}q_t,q_t)\le\frac{1}{\rho_m c_\zeta}\sum_{i=1}^{m-1}\frac{q(t_i,t_{i+1})^2}{(\Delta t_i)^\zeta}
\le\frac{2^{2+\zeta}M^2(m-1)^3}{\rho_m c_\zeta\nu^2}\,\varepsilon^{2-\zeta} \eqqcolon C_2\,\varepsilon^{2-\zeta}.
\]
Now, choose $\varepsilon_0\in(0,\eta/(m-1)]$ with $\nu C_2\varepsilon_0^{2-\zeta}\le\frac12$, and let
$\varepsilon\le\varepsilon_0$. For almost every $t\in\mathcal D_\varepsilon$ (namely
whenever $G_t>0$), Lemma~\ref{lem:block}(iii) gives
\[
(A_t^{-1}q_t,q_t)\le2(B_t^{-1}q_t,q_t)\le2C_2\varepsilon^{2-\zeta}.
\]
Moreover, Hadamard's inequality and $(A_t)_{ii}\le(B_t)_{ii}\le C_0(\Delta t_i)^\zeta$ give
$$G_t=\det A_t\le C_0^{m-1}\varepsilon^{(m-1)\zeta}.$$ Therefore, by \eqref{eq:Mp} with $p=1$,
\[
M_1(a)\ge\frac{|\mathcal D_\varepsilon|\;e^{-C_2\|a\|^2\varepsilon^{2-\zeta}}}{(2\pi C_0)^{\frac{(m-1)d}{2}}\,\varepsilon^{\frac{(m-1)\zeta d}{2}}}
=\frac{2^{1-m}|\Theta|}{(2\pi C_0)^{\frac{(m-1)d}{2}}}\;\varepsilon^{(m-1)(1-\frac{\zeta d}{2})}\,e^{-C_2\|a\|^2\varepsilon^{2-\zeta}}.
\]
Take $\varepsilon=\|a\|^{-2/(2-\zeta)}$, which is admissible for
$\|a\|\ge R\coloneq\varepsilon_0^{-(2-\zeta)/2}$. This gives
$$M_1(a)\ge c\,\|a\|^{-\gamma_{1,m}}$$ with
$$c=2^{1-m}|\Theta|(2\pi C_0)^{-(m-1)d/2}e^{-C_2},$$ which completes the proof.
\end{proof}

\begin{remark}\label{rem:4.3}
\begin{enumerate}[label=\textup{(\arabic*)},leftmargin=*,itemsep=2pt]
\item 
The lower bound in \eqref{eq:two-sided} shows 
 that the exponent $\gamma_{p,m}$ is optimal.
\item By Remark~\ref{rem:fBm}, the $d$-dimensional fractional Brownian motion $B^H$ with
$Hd<1$ satisfies the assumptions of Theorem~\ref{thm:4.2} with $\zeta=2H$. Hence,
\[
\E\big((T^{B^H}_m)^p\mid B^H(1)=a\big)\asymp\|a\|^{-p(m-1)\frac{1-Hd}{1-H}},\qquad\|a\|\to\infty .
\]
For $d=1$,  the exponent equals $p(m-1)$ for every $H\in(0,1)$. That is, for the one-dimensional fractional Brownian motion, the exponent does not depend on the Hurst parameter (and therefore, also the range of dependence on  history). 
\item For the one-dimensional Wiener process $w$ and $m=2$, the increment
$\eta^a(t_2)-\eta^a(t_1)$ has mean $ha$ and variance $h(1-h)$, where $h=t_2-t_1$. Hence,
\[
\E\big(T^w_2\mid w(1)=a\big)=\int_0^1\frac{1-h}{\sqrt{2\pi h(1-h)}}\;e^{-\frac{a^2h}{2(1-h)}}\,\differential{h} \sim \frac1{|a|},
\qquad |a|\to\infty,
\]
in accordance with $\gamma_{1,2}=1$. See also \cite{DorogovtsevIzyumtseva2015Properties}. 
\item Condition \eqref{eq:drift} enters the lower bound only through the monotonicity
of $t\mapsto\Cov(u_1(t),u_1(1))$.
\end{enumerate}
\end{remark}

\begin{remark}
  \label{rem:M_p_geometric_interpretation} 
  Self-intersection local times are intimately related to the geometry of the stochastic trajectories (\cite{LeGallWienersausage}). Therefore, as alluded to earlier in Remark~\ref{rem:geometric_interpretation}, our results naturally have a geometric counterpart. To be sure, the moment formula $M_p(a)$ also has a direct geometric interpretation. As before, denote by $\theta_t \in [0, \pi/2]$ the angle between $e = k(1,\cdot)1_{[0,1]}/\norm{k(1,\cdot)}$ and the linear span generated by $\{\Delta g(t_i) : \ i\in S\}$. Then, by the same reasoning as in Remark~\ref{rem:geometric_interpretation}, and Lemma~\ref{lem:block}(iii) we have 
  \begin{align*}
    (A_t^{-1}q_t,q_t) = \frac{ \cos^2(\theta_t)/\nu }{ \sin^2(\theta_t)} = \frac{\cot^2(\theta_t)}{\nu}, 
  \end{align*}
  whence the following geometric formula for the moments is immediate 
  \begin{align}\label{eq:Mp_geometric}
  \begin{aligned}
      M_p(a)&{}= {(2\pi)^{-\frac{p(m-1)d}{2}}} \int_{\Delta_m^p} \left(\vol{ \Delta g(t_i),\ i\in S} \sin(\theta_t)\right)^{-{d}} \\
&{}\quad\quad \quad\quad \quad \times 
\myExp{-\frac{\|a\|^2}{2\nu} \cot^2(\theta_t) } \differential{t_1}\cdots\differential{t_{pm}}. 
  \end{aligned}
\end{align}
From this representation of the moments in \eqref{eq:Mp_geometric}, we can see that the end-to-end distance enters the formula explicitly through the term $\myExp{-\frac{\|a\|^2}{2\nu} \cot^2(\theta_t) }$. Therefore, large end-to-end distance decreases the moments, unless $\cot(\theta_t)$ is small, \emph{i.e.}, unless $g(1)$ is nearly orthogonal to the linear span generated by the chords $\{\Delta g(t_i),\ i\in S\}$. 

\end{remark}

\section{Asymptotics of the conditional moments under strong local nondeterminism} 
\label{sec:asymptotics2}
Let $\{\eta^a(t) : t\in[0,1]\}$ be a planar Volterra Gaussian bridge conditioned to be $a$ at time $1$, defined in \eqref{eq:VGB}. We assume that $\eta^a$ is generated by a Volterra Gaussian process $u$ with a kernel $k$ as in \eqref{eq:stationary_kernel}. By  Lemma~\ref{lem:increments}, these conditions guarantee that the increments of $\eta^a$ behave like the increments of a Wiener process on small time intervals. As is the case for  the planar Wiener process (see \cite{Varadhan1969}, \cite{rosen1986renormalized}), the self-intersection local times $T^{\eta^a}_m$ cannot be expected to exist as limits of their approximations, and a renormalisation is needed. The Rosen renormalisation is given in Theorem \ref{thm:Fourier_Wiener}. The reason is that strongly locally nondeterministic Volterra Gaussian bridges have too many self-intersections near the diagonals of $\Delta_m$. Nevertheless, self-intersection local times for $\eta^a$ can be defined on any domain without the diagonal as the limit of approximations.   
For $\norm{a}>1$ and a control parameter $\alpha>0$, consider the domain $\Delta_2$ without the diagonal $t_1=t_2$ defined as follows:
 $$\Delta_{2,a,\alpha}=\{(t_1,t_2):\ 0\leqslant t_1\leqslant 1-\norm{a}^{-\alpha},\ t_1+\norm{a}^{-\alpha}\leqslant t_2\leqslant 1\},$$ and let
$$
T^{u}_{2,a,\alpha}=\int_{\Delta_{2,a,\alpha}}\delta_0(\eta^a(t_2)-\eta^a(t_1))\differential{t_1}\differential{t_2}.
$$
As in  Lemma~\ref{lem:conditional}, the conditional expectations given $u(1)=a$ are computed with respect to the regular conditional distribution of $u$ given $u(1)=a$, which is the law of $\eta^a$. Thus,  $\Eof{T^u_{2,a,\alpha}\mid u(1)=a}=\Eof{T^{\eta^a}_{2,a,\alpha}}$, where the right-hand side denotes the limit as $\ve\to0$ of the expectations of the approximations obtained by replacing the Dirac delta function $\delta_0$ with the Gaussian density $f_\ve$. The following statement describes the asymptotics of $\Eof{T^u_{2,a,\alpha}\mid u(1)=a}$ for large $\norm{a}.$
The following lemma shows that $\sigma^2(t_1,t_2)\coloneq\norm{(I-P)\Delta g(t_1,t_2)}^2$, the variance of each coordinate of $\eta^a(t_2)-\eta^a(t_1)$, vanishes only at the corner $(t_1,t_2)=(0,1)$ of $\Delta_2$, and it quantifies the rate. This is needed to justify the passage to the limit $\ve\to0$ in \eqref{eq:5.2} below, since $(0,1)\in\Delta_{2,a,\alpha}$. For the Brownian bridge,  one has $\sigma^2(t_1,t_2)=(t_2-t_1)(t_1+1-t_2)$, so the bound obtained is sharp.

\begin{lemma}\label{lem:sigma}
Let the kernel $k$ be as in \eqref{eq:stationary_kernel}. Then, there exists a constant $c_5>0$ such that
\[
\sigma^2(t_1,t_2)\geqslant c_5\,(t_2-t_1)\,(t_1+1-t_2), \qquad\text{for all }0\leqslant t_1<t_2\leqslant1.
\]
In particular, $\sigma^2>0$ on $\Delta_2\setminus\{(0,1)\}$, and $$\int_{\Delta_{2,a,\alpha}}\sigma^{-2}(t_1,t_2)\,\differential{t_1}\differential{t_2}<\infty$$ for all $\norm{a}>1$ and $\alpha>0$.
\end{lemma}
\begin{proof}[Proof of Lemma~\ref{lem:sigma}]
Write $h=t_2-t_1$, $\ve_1=t_1$, $\ve_2=1-t_2$, $K_1=\norm{k}_\infty$, $K_2=\norm{k'}_\infty$ and $\kappa_0^2=\int_0^1k^2(x)\differential{x}>0$. Let $\delta\in(0,1]$ be as in Lemma \ref{lem:increments}, and recall from its proof that $\int_0^xk^2 (r)\differential{r}\geqslant c_k x$ for $x\in[0,1]$, where $c_k=\frac34k^2(0)\delta$. Since $P$ is the orthogonal projection onto the span of $k(1-\cdot)$,
\begin{equation}\label{eq:sigma_min}
\sigma^2(t_1,t_2)=\min_{c\in\mbR}\norm{\Delta g(t_1,t_2)-c\,k(1-\cdot)}^2=\min_{c\in\mbR}\big(I_1(c)+I_2(c)+I_3(c)\big),
\end{equation}
where
\begin{align*}
I_1(c)&=\int_0^{t_1}\big(k(t_2-r)-k(t_1-r)-ck(1-r)\big)^2\differential{r},\\
I_2(c)&=\int_{t_1}^{t_2}\big(k(t_2-r)-ck(1-r)\big)^2\differential{r},\qquad\\
I_3(c)&=c^2\int_{t_2}^1k^2(1-r)\differential{r}=c^2\int_0^{\ve_2}k^2(x)\differential{x}.
\end{align*}

We proceed in steps. 
Note that, when $h< \delta$, by  Lemma~\ref{lem:increments}(1), $$\sigma^2\geqslant\frac12k^2(0)h\geqslant\frac12k^2(0)\,h\,(t_1+1-t_2),$$ because $t_1+1-t_2\leqslant1$.

Next, fix $\ve_0 \in (0, 1/4]$, and let $h \ge \delta$ and $\max\{\ve_1, \ve_2\} \ge \ve_0$. Note that the map $t\mapsto g(t)$ is continuous in $L^2([0,1])$ by Lemma \ref{lem:increments}(1). Therefore, $\sigma^2$ is also continuous on $\Delta_2$. If $0<s_1<\ldots<s_n\leqslant1$ and $\sum_jc_jg(s_j)=0$, then on $(s_{n-1},s_n)$ only $g(s_n)$ is nonzero, so $c_nk(s_n-r)=0$ for almost every $r\in(s_{n-1},s_n)$. Since $k(0)\neq0$ and $k$ is continuous, $c_n=0$, and by induction all $c_j$ vanish. Thus, we can conclude that  $g(s_1),\ldots,g(s_n)$ are linearly independent for distinct points of $(0,1]$. Since $g(0)=0$ and $\sigma^2(t_1,t_2)=0$ if and only if $g(t_2)-g(t_1)\in\operatorname{span}\{g(1)\}$, it follows (considering separately the cases $t_1>0$ and $t_1=0$) that $\sigma^2(t_1,t_2)>0$ unless $t_1=0$ and $t_2=1$. Consequently, for every $\ve_0\in(0,\frac12]$ there is $c(\ve_0)>0$ such that $\sigma^2\geqslant c(\ve_0)\geqslant c(\ve_0)\,h\,(t_1+1-t_2)$ on the compact set $\{t\in\Delta_2:\ h\geqslant\delta,\ \max(\ve_1,\ve_2)\geqslant\ve_0\}$.

Finally, let $h\ge \delta$ and $\max(\ve_1,\ve_2)<\ve_0$ with $\ve_0 \le 1/4$ so that $h=1-\ve_1 - \ve_2 \geqslant\frac12$. 
Now, put $$\theta=\frac{|k(0)|}{2(|k(0)|+|k(1)|)}\in(0,\frac12]$$ and fix $c\in\mbR$. On $[t_1,t_2]$, we write
\[
k(t_2-r)-c\,k(1-r)=(1-c)\,k(1-r)+\big(k(t_2-r)-k(1-r)\big).
\]
Since $t_2-r$ and $1-r$ lie in $[0,1]$ and differ by $\ve_2$, the mean value theorem gives
$|k(t_2-r)-k(1-r)|\leqslant K_2\ve_2$, and hence
$\norm{k(t_2-\cdot)-k(1-\cdot)}_{L^2([t_1,t_2])}\leqslant K_2\ve_2$.
Moreover, since $\ve_1+\ve_2<2\ve_0$,
\[
\int_{t_1}^{t_2}k^2(1-r)\differential{r}=\int_{\ve_2}^{1-\ve_1}k^2(r)\differential{r}
\geqslant\kappa_0^2-K_1^2(\ve_1+\ve_2)\geqslant\frac12\kappa_0^2
\qquad\text{if }\ \ve_0\leqslant\frac{\kappa_0^2}{4K_1^2}.
\]
By the reverse triangle inequality in $L^2([t_1,t_2])$ and the nonnegativity of the norm, we get
\[
I_2(c)^{1/2}\geqslant\Big(\frac{|1-c|\,\kappa_0}{\sqrt2}-K_2\ve_2\Big)_+,
\qquad\text{and hence, }\qquad
I_2(c)\geqslant\Big(\frac{|1-c|\,\kappa_0}{\sqrt2}-K_2\ve_2\Big)_+^2 .
\]


For $r\in[0,t_1],$ we have $|k(t_2-r)-k(1)|\leqslant K_2(\ve_1+\ve_2),$ $|k(t_1-r)-k(0)|\leqslant K_2\ve_1$ and $|k(1-r)-k(1)|\leqslant K_2\ve_1$, whence we get
\[
I_1(c)\geqslant\ve_1\Big(\big|(1-c)k(1)-k(0)\big|-K_2(2+|c|)(\ve_1+\ve_2)\Big)_+^2,\qquad I_3(c)\geqslant c_k\,c^2\ve_2 .
\]
If $|1-c|>\theta$, then $I_2(c)\geqslant\theta^2\kappa_0^2/8$ provided $K_2\ve_0\leqslant\theta\kappa_0/(2\sqrt2)$, and hence,  $$I_1+I_2+I_3\geqslant\frac{\theta^2\kappa_0^2}{8}\geqslant\frac{\theta^2\kappa_0^2}{8}(\ve_1+\ve_2).$$ If $|1-c|\leqslant\theta$, then $|c|\leqslant2$, $|(1-c)k(1)-k(0)|\geqslant|k(0)|-\theta|k(1)|\geqslant\frac12|k(0)|$ and $c^2\geqslant(1-\theta)^2\geqslant\frac14$. Therefore, provided $8K_2\ve_0\leqslant\frac14|k(0)|$,
\[
I_1+I_2+I_3\geqslant\frac{k^2(0)}{16}\,\ve_1+\frac{c_k}{4}\,\ve_2\geqslant\min\Big(\frac{k^2(0)}{16},\frac{c_k}{4}\Big)(\ve_1+\ve_2).
\]
Choosing $\ve_0=\min\{\frac14,\ \kappa_0^2/(4K_1^2),\ \theta\kappa_0/(2\sqrt2K_2),\ |k(0)|/(32K_2)\}$ (the conditions involving $K_2$ being void if $K_2=0$) and minimising over $c$ in the expression in \eqref{eq:sigma_min}, we obtain $\sigma^2\geqslant c'(\ve_1+\ve_2)\geqslant c'\,h\,(t_1+1-t_2)$, with the constant $c'>0$ depending only on $k$. 

This proves the inequality with $c_5=\min\{\frac12k^2(0),c(\ve_0),c'\}$.

Finally, on $\Delta_{2,a,\alpha}$, we have $h\geqslant\norm{a}^{-\alpha}$. Therefore,  $$\sigma^{-2}\leqslant c_5^{-1}\norm{a}^{\alpha}(t_1+1-t_2)^{-1},$$ and $$\int_{[0,1]^2}(x+y)^{-1}\differential{x}\differential{y}=2\log2<\infty.$$
\end{proof}

\subsection{Main theorem on asymptotics under strong local nondeterminism}
We are now ready to state and prove the main result of this section. The following theorem gives the asymptotics of the conditional moments of the self-intersection local times of a strongly locally nondeterministic Volterra Gaussian bridge, conditioned to be $a$ at time $1$, as $\norm{a}\to\infty$. In particular, it shows that the threshold $\alpha=2$ is critical for the asymptotic behaviour of the conditional moments. The conditional moment of the self-intersection local times of the bridge is finite for all $\alpha>0$, but it decays exponentially fast for $\alpha<2$, converges to a positive constant for $\alpha=2$, and grows logarithmically for $\alpha>2$.

\begin{theorem}\label{thm:5.1}
Let the kernel $k$ be as in \eqref{eq:stationary_kernel}, and put
\[
\psi(t)=\Cov(u_1(t),u_1(1))=\int_0^tk(t-r)k(1-r)\,\differential{r},\qquad\nu=\psi(1)=\|k(1,\cdot)\|^2 .
\]
Assume that there exists a constant $c_0>0$ such that
\begin{equation}\label{eq:5.1}
|\psi(t)-\psi(s)|\ge c_0(t-s),\qquad 0\le s<t\le1.
\end{equation}
Then, the following statements hold as $\|a\|\to\infty$.
\begin{enumerate}[label=\textup{(\arabic*)},leftmargin=*,itemsep=3pt]
\item If $0<\alpha<2$, there exist constants $0<c\le C<\infty$ such that, for all
sufficiently large $\|a\|$,
\[
c\,e^{-C\|a\|^{2-\alpha}}\le\E\big(T^u_{2,a,\alpha}\mid u(1)=a\big)\le C\,e^{-c\|a\|^{2-\alpha}}.
\]
\item If $\alpha=2$, then
\[
\lim_{\|a\|\to\infty}\E\big(T^u_{2,a,2}\mid u(1)=a\big)
=\frac{1}{2\pi k^2(0)}\int_0^1E_1\Big(\frac{\psi'(t)^2}{2\nu^2k^2(0)}\Big)\,\differential{t} \in(0,\infty),
\]
where $E_1(z)=\int_z^\infty x^{-1}e^{-x}\,\differential{x}$ is the exponential integral function.
\item If $\alpha>2$, then
$
\E\big(T^u_{2,a,\alpha}\mid u(1)=a\big)=\frac{\alpha-2}{2\pi k^2(0)}\,\log\|a\|+O(1).
$
\end{enumerate}
\end{theorem}
\begin{proof}[Proof of Theorem~\ref{thm:5.1}]
For $0\le t_1<t_2\le1$, write $h=t_2-t_1$, $\Delta g=\Delta g(t_1,t_2)$ and
$\sigma^2=\sigma^2(t_1,t_2)=\|(I-P)\Delta g(t_1, t_2)\|^2$. Recall that
$q=q(t_1,t_2)=(\psi(t_2)-\psi(t_1))/\nu$. For $\ve>0$ and $(t_1,t_2)\in\Delta_2$ with $\sigma^2>0$, $\Eof{f_\ve(\eta^a(t_2)-\eta^a(t_1))}$ is the integral of $f_\ve$ against the Gaussian density with mean $qa$ and covariance matrix $\sigma^2I_2$. Naturally, it converges to the value of that density at $0$ as $\ve\to0$ and is bounded by $(2\pi\sigma^2)^{-1}$, which is integrable over $\Delta_{2,a,\alpha}$ by Lemma \ref{lem:sigma}. Hence, by Fubini's theorem and dominated convergence,
\begin{equation}\label{eq:5.2}
\E\big(T^u_{2,a,\alpha}\mid u(1)=a\big)=\iint_{\Delta_{2,a,\alpha}}p_a(t_1,t_2)\,\differential{t_1}\differential{t_2},
\qquad p_a=\frac{1}{2\pi\sigma^2}\exp\Big(-\frac{\|a\|^2q^2}{2\sigma^2}\Big).
\end{equation}

By the decomposition in the proof of
Lemma~\ref{lem:increments} (1), $\sigma^2=\int_0^hk^2(x)\,\differential{x}+O(h^2)$. Moreover,
$\int_0^hk^2(x)\,\differential{x}=k^2(0)h+O(h^2)$ because $k^2$ is Lipschitz, and both $O$-terms
are uniform in $t_1$. Hence, there are $\delta\in(0,\frac12]$ and $C_1<\infty$ such that
\begin{equation}\label{eq:sigma}
\tfrac12k^2(0)h\le\sigma^2\le2k^2(0)h\quad\text{and}\quad
\Big|\frac1{\sigma^2}-\frac1{k^2(0)h}\Big|\le C_1\qquad\text{whenever } h\le\delta .
\end{equation}
By Lemma~\ref{lem:increments} (1), $\sigma^2\le\|\Delta g\|^2\le C_3h\le C_3$ for all $t_1<t_2$. By the inequality in 
\eqref{eq:5.1} and since $\psi\in C^1([0,1])$, we have 
\begin{equation}\label{eq:q}
\frac{c_0}{\nu}\,h\le|q|\le C_2h,\qquad C_2=\frac{\|\psi'\|_\infty}{\nu}.
\end{equation}
Consequently, with $c_3=c_0^2/(2\nu^2k^2(0))$ and $C_4=2C_2^2/k^2(0)$,
\begin{equation}\label{eq:expo}
c_3\|a\|^2h\le\frac{\|a\|^2q^2}{\sigma^2}\le C_4\|a\|^2h\qquad\text{whenever } h\le\delta .
\end{equation}
Finally, let $h>\delta$. Then, $q^2\ge b\coloneq c_0^2\delta^2/\nu^2$. The function
$x\mapsto x^{-1}e^{-y/x}$ is increasing on $(0,y]$, and
$\sigma^2\le C_3\le\frac12\|a\|^2q^2$ when $\|a\|^2\ge2C_3/b$. Hence,
$$p_a\le(2\pi C_3)^{-1}\myExp{-b\|a\|^2/(2C_3)}$$ on $\{h>\delta\}$. Since
$\Delta_{2,a,\alpha}$ has area at most $\frac12$, the part of \eqref{eq:5.2} over
$\{h>\delta\}$ is at most
\begin{equation}\label{eq:far}
r_a\coloneq\frac{1}{4\pi C_3}\,e^{-b\|a\|^2/(2C_3)}.
\end{equation}
By \eqref{eq:sigma}, \eqref{eq:expo}
and \eqref{eq:far}, we get 
\[
\E\big(T^u_{2,a,\alpha}\mid u(1)=a\big) - r_a \le\int_{\|a\|^{-\alpha}}^\infty\frac{e^{-c_3\|a\|^2h/2}}{\pi k^2(0)h}\,\differential{h}
=\frac{E_1(\frac{c_3}{2}\|a\|^{2-\alpha})}{\pi k^2(0)}
\le\frac{2\,e^{-\frac{c_3}{2}\|a\|^{2-\alpha}}}{\pi k^2(0)c_3\|a\|^{2-\alpha}},
\]
where we used $E_1(z)\le \myExp{-z}/z$ (which follows from \cite[\S 6.8]{DLMF}). Since $r_a\le C\myExp{-c\|a\|^2}\le C\myExp{-c\|a\|^{2-\alpha}}$,
the upper bound follows.

Let the end-to-end distance $\|a\|$ be so large that $2\|a\|^{-\alpha}\le\delta$. Note that the set
$
\{(t_1,t_1+h):\ t_1\in[0,\tfrac12],\ h\in[\|a\|^{-\alpha},2\|a\|^{-\alpha}]\}
$
is contained in $\Delta_{2,a,\alpha}$ and has area $\frac12\|a\|^{-\alpha}$. On this
set, by \eqref{eq:sigma} and \eqref{eq:expo}, we have 
\[
p_a\ge(4\pi k^2(0)h)^{-1}e^{-C_4\|a\|^2h/2}\ge(8\pi k^2(0))^{-1}\|a\|^{\alpha}e^{-C_4\|a\|^{2-\alpha}}.
\]
Hence, we get the lower estimate  $$\E(T^u_{2,a,\alpha}\mid u(1)=a)\ge(16\pi k^2(0))^{-1}e^{-C_4\|a\|^{2-\alpha}}.$$
 By \eqref{eq:far}, the part of \eqref{eq:5.2} over
$\{h>\delta\}$ tends to $0$. In the remaining part,  we substitute $h=x\|a\|^{-2}$:
\[
\iint_{\Delta_{2,a,2}\cap\{h\le\delta\}}p_a(t_1, t_2)\,\differential{t_1}\differential{t_2}
=\int_0^1\!\!\int_1^\infty\mathbf 1_{\{x\le\|a\|^2\min(\delta,1-t_1)\}}\,\Phi_a(t_1,x)\,\differential{x}\,\differential{t_1},
\]
where $$\Phi_a(t_1,x)=\|a\|^{-2}p_a(t_1,t_1+x\|a\|^{-2}).$$ With $h=x\|a\|^{-2}$,
\[
\Phi_a(t_1,x)=\frac{1}{2\pi x}\,\frac{h}{\sigma^2}\,\exp\Big(-\frac x2\Big(\frac qh\Big)^2\frac{h}{\sigma^2}\Big).
\]
Fix $t_1\in[0,1)$ and $x\ge1$, and let $\|a\|\to\infty$, so that $h\to0$. Then,
$\sigma^2/h\to k^2(0)$ by \eqref{eq:sigma}, and $q/h\to\psi'(t_1)/\nu$ because
$\psi\in C^1([0,1])$. Hence,
\[
\Phi_a(t_1,x)\to(2\pi k^2(0)x)^{-1}\exp\big(-x\psi'(t_1)^2/(2\nu^2k^2(0))\big).
\]
Moreover, $\Phi_a(t_1,x)\le(\pi k^2(0)x)^{-1}\myExp{-c_3x/2}$ by \eqref{eq:sigma} and
\eqref{eq:expo}, and this bound is integrable on $[0,1]\times[1,\infty)$. By
dominated convergence, we get 
\begin{align*}
\lim_{\|a\|\to\infty}\E\big(T^u_{2,a,2}\mid u(1)=a\big)
& =\int_0^1\!\!\int_1^\infty\frac{\myExp{-x\psi'(t)^2/(2\nu^2k^2(0))}}{2\pi k^2(0)\,x}\,\differential{x}\,\differential{t}\\
& =\frac{1}{2\pi k^2(0)}\int_0^1 E_1\Big(\frac{\psi'(t)^2}{2\nu^2k^2(0)}\Big)\differential{t} .
\end{align*}
It is straightforward to see that the limit is positive. It is also finite because $|\psi'|\ge c_0$ by \eqref{eq:5.1}.

By \eqref{eq:sigma},
$(2\pi\sigma^2)^{-1}\le(2\pi k^2(0)h)^{-1}+C_1/(2\pi)$ for $h\le\delta$. Together with
\eqref{eq:expo} and \eqref{eq:far}, this gives
\[
\E\big(T^u_{2,a,\alpha}\mid u(1)=a\big)\le\frac{1}{2\pi k^2(0)}\int_{\|a\|^{-\alpha}}^\infty\frac{e^{-c_3\|a\|^2h/2}}{h}\,\differential{h} +\frac{C_1\delta}{2\pi}+r_a
=\frac{E_1(\frac{c_3}{2}\|a\|^{2-\alpha})}{2\pi k^2(0)}+O(1).
\]
Since $E_1(z)=\log\frac1z+O(1)$ as $z\downarrow0$ (see \cite[\S 6.6]{DLMF}), and
$\log(\frac{2}{c_3}\|a\|^{\alpha-2})=(\alpha-2)\log\|a\|+O(1)$, the upper bound follows.

Let $\|a\|$ be so large that $\|a\|^{-2}\le\min\{\delta,1/(C_1k^2(0))\}$.
The set
$
\{(t_1,t_1+h):\ t_1\in[0,1-\|a\|^{-2}],\ h\in[\|a\|^{-\alpha},\|a\|^{-2}]\}
$
is contained in $\Delta_{2,a,\alpha}$. On this set,
$(2\pi\sigma^2)^{-1}\ge(2\pi k^2(0)h)^{-1}-C_1/(2\pi)\ge0$ by \eqref{eq:sigma}.
Moreover, by \eqref{eq:expo} and $e^{-y}\ge1-y$,
$$
\exp\big(-\|a\|^2q^2/(2\sigma^2)\big)\ge1-\tfrac{C_4}{2}\|a\|^2h .
$$
Hence,
\begin{align*}
\E\big(T^u_{2,a,\alpha}\mid u(1)=a\big)
&\ge(1-\|a\|^{-2})\int_{\|a\|^{-\alpha}}^{\|a\|^{-2}}\Big(\frac{1}{2\pi k^2(0)h}-\frac{C_1}{2\pi}\Big)\Big(1-\frac{C_4}{2}\|a\|^2h\Big)\differential{h}\\
&\ge(1-\|a\|^{-2})\,\frac{(\alpha-2)\log\|a\|}{2\pi k^2(0)}-\frac{C_4}{4\pi k^2(0)}-\frac{C_1}{2\pi}\|a\|^{-2}.
\end{align*}
For the second inequality, expand the product, use $\|a\|^2(\|a\|^{-2}-\|a\|^{-\alpha})\le1$
and $1-\|a\|^{-2}\le1$, and note that the cross term is nonnegative. The right-hand side equals
$\frac{\alpha-2}{2\pi k^2(0)}\log\|a\|-O(1)$, which completes the proof.
\end{proof}

\begin{remark}\label{rem:5.2}
\begin{enumerate}[label=\textup{(\arabic*)},leftmargin=*,itemsep=2pt]
\item Since $k\in C^1([0,1])$, we have $\psi\in C^1([0,1])$ with
\[
\psi'(t)=k(0)k(1-t)+\int_0^tk'(t-r)k(1-r)\,\differential{r} .
\]
Condition \eqref{eq:5.1} is equivalent to $\min_{[0,1]}|\psi'|>0$.

If $|k|$ is nondecreasing on $[0,1]$, then $|k|\ge|k(0)|>0$, so $k$ has constant
sign and $kk'\ge0$. Hence, $|\psi'(t)|\ge k^2(0)$, \emph{i.e.}, \eqref{eq:5.1} holds with
$c_0=k^2(0)$. This covers the Wiener process.

For $k(x)=1-x$, the condition \eqref{eq:5.1}  fails, since $\psi'(0)=0$.
\item For the Wiener process, $\psi(t)=t$ and $\nu=k(0)=1$. The limit in Theorem~\ref{thm:5.1}(2) is
$E_1(\frac12)/(2\pi)\approx0.0891$. For Theorem~\ref{thm:5.1}(3),  a direct computation gives
\[
\E\big(T^w_{2,a,\alpha}\mid w(1)=a\big)=\frac{\alpha-2}{2\pi}\log\|a\|+\frac{\log2-\gamma_{\mathrm E}}{2\pi}+o(1),
\]
where $\gamma_{\mathrm E}$ is  Euler's constant. See also \cite{DorogovtsevIzyumtseva2015Properties}.
\item Condition \eqref{eq:5.1} is used only for the upper bounds and for the
domination in Theorem~\ref{thm:5.1}(2). The lower bounds in Theorem~\ref{thm:5.1}(1) and Theorem~\ref{thm:5.1}(3) use only
$\psi\in C^1([0,1])$ and Lemma~\ref{lem:increments}.
\item The threshold $\alpha=2$ has an interpretation. For $h=t_2-t_1\ll\|a\|^{-2}$,
the mean $q(t_1,t_2)a$ of $\eta^a(t_2)-\eta^a(t_1)$ is of order $h\|a\|$, which is
negligible compared with its standard deviation $\approx|k(0)|h^{1/2}$. On these time
scales the bridge behaves like a planar Brownian motion with diffusion coefficient
$k^2(0)$.
For
$h\gg\|a\|^{-2}$, the drift dominates and the contribution is negligible.
\end{enumerate}
\end{remark}









\bibliographystyle{amsplain}

\bibliography{refs}

\end{document}